\documentclass[twoside,11pt,reqno]{amsart}
\usepackage{amsmath,amssymb,amscd,mathrsfs,amscd}
\usepackage{graphics,verbatim}
\usepackage{todonotes}
\usepackage{hyperref}

\let\oldsection=\section

\newcommand{\p}[1]{\ensuremath{\overline{#1}}}
\newcommand{\losemi}{{\otimes \kern -.78em \ltimes}}
\newcommand{\rosemi}{{\otimes \kern -.78em \rtimes}}

\newcommand{\Hom}{\ensuremath{\operatorname{Hom}}}

\newcommand{\Ind}{\ensuremath{\operatorname{ind}}}
\newcommand{\Res}{\ensuremath{\operatorname{Res}}}

\newcommand{\Ext}{\operatorname{Ext}}

\newcommand{\0}{\bar 0}
\newcommand{\1}{\bar 1}
\newcommand{\Z}{\mathbb{Z}}
\newcommand{\C}{\mathbb{C}}

\newcommand{\gl}{\ensuremath{\mathfrak{gl}}}
\newcommand{\g}{\ensuremath{\mathfrak{g}}}

\newcommand{\f}{\ensuremath{\mathfrak{f}}}

\newcommand{\rk}{\operatorname{rank}}
\newcommand{\X}{\mathcal{X}}

\newcommand{\fg}{\ensuremath{\mathfrak{g}}}
\newcommand{\fb}{\ensuremath{\mathfrak{b}}}

\newcommand{\fh}{\ensuremath{\mathfrak{h}}}
\newcommand{\ff}{\ensuremath{\f}}

\newcommand{\fp}{\ensuremath{\mathfrak{p}}}
\newcommand{\ft}{\ensuremath{\mathfrak{t}}}

\newcommand{\s}{\sigma}

\newcommand{\atyp}{\ensuremath{\operatorname{atyp}}}
\newcommand{\V}{\mathcal{V}}
\newcommand{\HH}{\operatorname{H}}
\newcommand{\F}{\mathcal{F}}

\newcommand{\B}{\mathcal{B}}
\newcommand{\la}{\lambda}

\newcommand{\R}{\ensuremath{\mathbb{R}}}
\newcommand{\md}{\operatorname{\mathsf{d}}}

\def\Label{\label}

\newtheorem{theorem}{Theorem}[subsection]

\makeatletter\let\c@fact\c@theorem\makeatother

\makeatletter\let\c@note\c@theorem\makeatother

\newtheorem{lemma}{Lemma}[subsection]
\makeatletter\let\c@lemma\c@theorem\makeatother

\makeatletter\let\c@alg\c@theorem\makeatother

\newtheorem{remark}{Remark}[subsection]
\makeatletter\let\c@remark\c@theorem\makeatother

\makeatletter\let\c@example\c@theorem\makeatother

\newtheorem{prop}{Proposition}[subsection]
\makeatletter\let\c@prop\c@theorem\makeatother

\makeatletter\let\c@conj\c@theorem\makeatother

\makeatletter\let\c@cor\c@theorem\makeatother

\makeatletter\let\c@defn\c@theorem\makeatother

\numberwithin{equation}{subsection}

\usepackage[capitalise]{cleveref}
\crefname{theorem}{Theorem}{Theorems}
\crefname{fact}{Fact}{Facts}
\crefname{note}{Note}{Notes}
\crefname{lemma}{Lemma}{Lemmas}
\crefname{alg}{Algorithm}{Algorithms}
\crefname{remark}{Remark}{Remarks}
\crefname{example}{Example}{Examples}
\crefname{prop}{Proposition}{Propositions}
\crefname{conj}{Conjecture}{Conjectures}
\crefname{cor}{Corollary}{Corollaries}
\crefname{defn}{Definition}{Definitions}
\crefname{equation}{}{}

\begin{document}
\title{Complexity for modules over the classical Lie superalgebra ${\mathfrak gl}(m|n)$}

\author{Brian D. Boe }
\address{Department of Mathematics \\
          University of Georgia \\
          Athens, GA 30602}
\email{brian@math.uga.edu}
\author{Jonathan R. Kujawa}
\address{Department of Mathematics \\
          University of Oklahoma \\
          Norman, OK 73019}
\thanks{Research of the second author was partially supported by NSF grant
DMS-0734226 and NSA grant H98230-11-1-0127}\
\email{kujawa@math.ou.edu}
\author{Daniel K. Nakano}
\address{Department of Mathematics \\
          University of Georgia \\
          Athens, GA 30602}
\thanks{Research of the third author was partially supported by NSF
grant  DMS-1002135}\
\email{nakano@math.uga.edu}
\date{\today}

\begin{abstract} Let ${\mathfrak g}={\mathfrak g}_{\0}\oplus {\mathfrak g}_{\1}$ be a classical Lie superalgebra and 
${\mathcal F}$ be the category of finite dimensional ${\mathfrak g}$-supermodules which are completely reducible over the 
reductive Lie algebra ${\mathfrak g}_{\0}$. In \cite{BKN3}, the authors demonstrated that for any module $M$ in 
${\mathcal F}$ the rate of growth of the minimal projective resolution (i.e., the complexity of $M$) is bounded 
by the dimension of ${\mathfrak g}_{\1}$. In this paper we compute the complexity of the simple modules and the 
Kac modules for the Lie superalgebra $\mathfrak{gl}(m|n)$. In both cases we show that the complexity is related to the atypicality 
of the block containing the module. 
\end{abstract}

\maketitle

\section{Introduction}\label{S:intro}  

\subsection{} Let ${\mathfrak g}={\mathfrak g}_{\0}\oplus {\mathfrak g}_{\1}$ be a classical Lie superalgebra over the complex numbers, ${\mathbb C}$. 
For classical Lie superalgebras ${\mathfrak g}_{\0}$ is a reductive Lie algebra. An important category of ${\mathfrak g}$-supermodules is the 
category ${\mathcal F}:={\mathcal F}_{(\g,\g_{\0})}$ of finite dimensional ${\mathfrak g}$-supermodules which are completely reducible over ${\mathfrak g}_{\0}$. 
The category ${\mathcal F}$ has enough projectives and is in general not semisimple. In \cite{BKN3}, the authors showed that 
(i) ${\mathcal F}$ is a self-injective category (meaning that a module being projective is equivalent to the module being injective), and 
(ii) every module in ${\mathcal F}$ admits a projective resolution which has polynomial rate of growth. For a module 
$M\in {\mathcal F}$, the complexity $c_{\mathcal F}(M)$ is the rate of growth of the minimal projective resolution of $M$. 
In \cite{BKN3}, it was proved by constructing an explicit Koszul type resolution that $c_{\mathcal F}(M)\leq \dim \g_{\1}$ for all $M\in {\mathcal F}$. 

It is well known that if $G$ is a finite group scheme then the category of rational modules for $G$ satisfies the same 
properties (i) and (ii) as given above. In this context the complexity of a module was first introduced by Alperin \cite{Al} in 1977. 
By using the fact that the cohomology ring for $G$ is finitely generated (cf.\ \cite{FS}), one can construct the (cohomological) 
support variety ${\mathcal V}_{G}(M)$ of a module $M$, whose dimension coincides with the complexity $c_{G}(M)$. This realization 
allows one to use geometric methods to compute the complexities of important classes of modules (see \cite{NPV, DNP, UGA3, HN}). 

The elusive ingredient for the superalgebra category ${\mathcal F}$ is a ``support variety'' theory which enables one to compute the 
complexity of modules in ${\mathcal F}$. In \cite{BKN3}, it was shown that there is a formula in terms of rates of growth of cohomology 
groups which realizes the complexity. The main goal of this paper will be to show how to compute the complexity for important classes of modules for $\gl (m|n)$. In particular we will show for $\mathfrak{gl}(m|n)$ that 
\begin{equation}\label{E:Kaccomplexity}
c_{\mathcal F}(K(\lambda))=(m+n)\atyp(\lambda)-\atyp(\lambda)^{2},
\end{equation}
which equals the dimension of the variety of $m \times n$ matrices of rank at most $\atyp(\lambda)$; and 
\begin{equation}\label{E:simplecomplexity}
c_{\mathcal F}(L(\lambda))=(m+n)\atyp(\lambda)-\atyp(\lambda)^{2}+\atyp(\lambda).
\end{equation}
Here $K(\lambda)$ (resp.\ $L(\lambda)$) is the Kac (resp.\ simple) 
module of highest weight $\lambda$, and $\text{atyp}(\lambda)$ is the atypicality of the weight $\lambda$ as defined by Kac and Wakimoto. 
Unlike the case with finite group schemes, our calculations show that the complexity is not invariant under equivalence of blocks. The proofs demonstrating these 
calculations employ a myriad of deep results, both known and new, about the category ${\mathcal F}$.  It is also worth noting that the formulas given in \cref{E:Kaccomplexity} and \cref{E:simplecomplexity} have the following remarkable geometric interpretation.  For a $\gl (m|n)$-module $M$ let $\X_{M}$ denote the associated variety defined by Duflo and Serganova \cite{dufloserganova},  and $\V_{(\fg , \fg_{\0})}(M)$ the support variety of \cite{BKN1}.  Then if $X(\lambda)$ is a Kac, dual Kac, or simple $\gl (m|n)$-module, we have 
\[
c_{\mathcal F}(X(\lambda)) = \dim \X_{X(\lambda)} + \dim \V_{(\fg , \fg_{\0})}(X(\lambda)).
\]

The paper is organized as follows. In \cref{S:preliminaries}, we set up the basic conventions for classical Lie superalgebras 
and in particular the Lie superalgebra $\mathfrak{gl}(m|n)$. We then introduce a support variety theory and relate this to the calculation of complexity for modules over the ``parabolic'' subalgebras of $\mathfrak{gl}(m|n)$ which are in turn 
later used (in \cref{S:KacComplexity})  with Serganova's equivalences between blocks of $\F$ to determine a lower bound on the complexity of Kac modules. In order to establish the 
upper bound we invoke results on the dimensions of projective modules in ${\mathcal F}$ developed in 
\cref{S:projmodulebound}. In \cref{S:KacComplexity}, we establish the aforementioned formula on the complexity of Kac (and dual Kac) modules. 

The remainder of the paper is devoted to computing the complexity of simple modules in ${\mathcal F}$. This computation is much more complicated because there is no known support variety theory for modules which measures complexity for modules in ${\mathcal F}$. In \cref{S:polytope}, we begin by establishing a lower bound on the dimension of projective indecomposable modules 
by using Ehrhart's theorem on counting lattice points in a polytope. The establishment of this bound also uses a combinatorial bijection on highest weights introduced by Su and Zhang.  Next, in \cref{S:simplecomplexity}, Serganova's recent verification of the generalized 
Kac-Wakimoto Conjecture for $\gl(m|n)$ is employed to reduce to a specific simple $\gl (m|n)$-module of atypicality $k$.  We then 
apply Brundan's deep results on the  characters and extensions of simple modules in ${\mathcal F}$ which prove that 
${\mathcal F}$ is a highest weight category having a Kazhdan-Lusztig theory to estimate the upper and lower bound 
for the complexity of simples via properties of Kazhdan-Lusztig polynomials. These results in conjunction with our established 
results in \cref{S:polytope} allow us to complete the calculation. Finally in \cref{S:zcomplexity}, we introduce a new 
numerical invariant of ${\mathcal F}$ which remains the same under equivalence of categories.  We also give evidence that this invariant is closely related to a detecting subalgebra of $\gl (m|n)$ previously introduced by the authors. 

We acknowledge Zongzhu Lin for a useful discussion which led to the results in the last section of the paper.

\section{Preliminaries} \label{S:preliminaries}

\subsection{Classical Lie Superalgebras} 

We will use the notation and conventions developed in 
\cite{BKN2, BKN1}. For more details we refer the reader to \cite[Section 2.1]{BKN2}. 

We will work over the complex numbers $\C$ throughout this paper. Let ${\mathfrak g}$ be a 
\emph{Lie superalgebra}; that is, a $\Z_{2}$-graded vector space 
$\g=\g_{\0}\oplus \g_{\1}$ with a bracket operation $[\;,\;]:\g \otimes \g \to \g$ which preserves the 
$\Z_{2}$-grading and satisfies graded versions of the usual Lie bracket axioms.  The subspace 
$\g_{\0}$ is a Lie algebra under the bracket and ${\mathfrak g}_{\1}$ is a $\g_{\0}$-module. A finite dimensional Lie superalgebra 
${\mathfrak g}$ is called \emph{classical} if there is a connected reductive algebraic group $G_{\0}$ such 
that $\operatorname{Lie}(G_{\0})=\g_{\0},$ and an action of $G_{\0}$ on $\g_{\1}$ which differentiates to 
the adjoint action of $\g_{\0}$ on $\g_{\1}$.\footnote{Unlike in Kac's original definition \cite{Kac1}, we do \emph{not} require a classical Lie superalgebra to be simple.} 
If $\g$ is classical, then $\g_{\1}$ is semisimple as a $\g_{\0}$-module.  
A \emph{basic classical} Lie superalgebra is a classical Lie superalgebra with a nondegenerate invariant supersymmetric 
even bilinear form (cf.\ \cite{Kac1}).  

Let $U(\g)$ be the universal enveloping superalgebra of ${\mathfrak g}$. The objects of the category of $\g$-supermodules 
are all $\Z_{2}$-graded left $U(\g)$-modules.  To describe the morphisms we first recall that if $M$ and $N$ are $\Z_{2}$-graded, then $\Hom_{\C}(M,N)$ is naturally 
$\Z_{2}$-graded by setting $\p{f}=r \in \Z_{2}$ if $f(M_{i}) \subseteq N_{i+r}$ for $i \in \Z_{2}$.  Here and elsewhere we write $\p{v} \in \Z_{2}$ for the degree of a homogeneous element $v$ of a $\Z_{2}$-graded vector space.  
We use the convention that we only state conditions for a homogenous element, with the general case given by linearity.  For $\fg$-supermodules $M$ and $N$ a homogeneous $\fg$-morphism $f : M \to N$ is a homogeneous linear map which satisfies 
\[
f(xm)=(-1)^{\p{f} \; \p{x}}xf(m)
\] for all homogeneous $x \in \fg$.  
Given ${\mathfrak g}$-supermodules $M$ and $N$ one can use the 
antipode and coproduct of $U({\mathfrak g})$ to define a ${\mathfrak g}$-supermodule 
structure on the contragradient dual $M^{*}$ and the tensor product $M\otimes N$. 

A supermodule is 
\emph{finitely semisimple} if it decomposes into a direct sum of finite dimensional simple supermodules.  
We write $\F = \F_{(\fg , \fg_{\0})}$ for the full subcategory of all finite dimensional $\fg$-supermodules which are finitely semisimple when viewed as $\fg_{\0}$-supermodules by restriction.  
As only supermodules will be considered in this paper, we will from now on use the term ``module" with the understanding that the prefix ``super'' is implicit.

\subsection{Complexity} Let $\{{V}_{t} \mid t\in {\mathbb N}\}=
\{{V}_{\bullet}\}$ be a sequence of finite dimensional ${\mathbb C}$-vector spaces. The 
\emph{rate of growth} of ${V}_{\bullet}$, $r({V}_{\bullet})$, is the smallest nonnegative integer $c$ such that 
there exists a constant $C>0$ with $\dim V_{t}\leq C\cdot t^{c-1}$ for all $t$. If no such integer exists then ${V}_{\bullet}$ is said to have \emph{infinite rate of 
growth}. Let $M\in {\mathcal F}$ and $P_{\bullet}\twoheadrightarrow M$ be a minimal projective resolution for $M$. Following 
Alperin \cite{Al}, we define the \emph{complexity} of $M$ to be $c_{\mathcal F}(M):=r(P_{\bullet} )$.  The following theorem was 
proved by the authors in \cite[Proposition 2.8.1]{BKN3} and provides a characterization of the complexity via rates of growth of extension groups in ${\mathcal F}$. 
This characterization will be important for our computational purposes. 

\begin{prop}\label{P:complexityext}
Let $\fg$ be a classical Lie superalgebra, and let $M$ be an object in $\F$. Then
$$
c_{\F}(M)= r\left( \Ext^{\bullet}_{(\g,\g_{\0})}(M,\bigoplus  S^{\,\dim P(S)})  \right)
$$
where the sum is over all simple modules $S$ in $\F$, and $P(S)$ is the projective cover of $S$.
\end{prop}

Note that $\Ext^{\bullet}_{(\fg , \fg_{\0})}(M,N)$ denotes relative cohomology for the pair $(\fg , \fg_{\0})$. When both $M$ and $N$ are objects of $\F_{(\fg, \fg_{\0})}$, then by \cite[Theorem 2.5.1]{BKN1} we have  
\begin{equation}\label{E:relcohom}
\Ext^{d}_{(\fg , \fg_{\0})}(M,N) \cong \Ext^{d}_{\F_{(\fg , \fg_{\0})}}(M,N)
\end{equation}
for all $d \geq 0$.

\subsection{Type I Lie Superalgebras} A Lie superalgebra is said to be of {\em Type I} if
it admits a $\Z$-grading ${\mathfrak g}=\g_{-1}\oplus {\g}_{0}\oplus {\g}_{1}$ concentrated
in degrees $-1,$ $0,$ and $1$ with ${\mathfrak g}_{\0}={\mathfrak g}_{0}$ and
${\mathfrak g}_{\1}=\g_{-1}\oplus {\g}_{1}$ and if the bracket respects this grading. Otherwise, ${\mathfrak g}$ is of Type II.
Examples of Type I Lie superalgebras include: $\mathfrak{gl}(m|n)$ and the simple Lie superalgebras of
types $A(m,n)$, $C(n)$ and $P(n)$.

The simple modules for ${\mathfrak g}$, a Type I classical Lie superalgebra, can be constructed in the
following way. Let ${\mathfrak t}$ be a Cartan subalgebra of ${\mathfrak g}_{0}$ and
$X^{+}_{0} \subseteq {\mathfrak t}^{*}$ be the set of dominant integral weights for $\fg_{0}$ (with
respect to a fixed Borel subalgebra of $\fg_{0}$).  For $\lambda \in X^{+}_{0},$ let $L_{0}(\lambda)$
be the simple finite dimensional ${\mathfrak g}_{0}$-module of highest weight $\lambda$. Set
$$
{\mathfrak p}^{+}= {\mathfrak g}_{0} \oplus {\mathfrak g}_{1}
\qquad \text{and} \qquad  {\mathfrak p}^{-}= {\mathfrak g}_{0} \oplus {\mathfrak g}_{-1} .
$$
Since ${\mathfrak g}$ is a Type I Lie superalgebra ${\mathfrak g}_{\pm 1}$ is an abelian
ideal of ${\mathfrak p}^{\pm}$.  We can therefore
view $L_{0}(\lambda)$ as a simple ${\mathfrak p}^{\pm}$-module via inflation.  In this way we obtain a complete set of finite dimensional 
simple modules for $\fp^{\pm}$.

For each $\lambda\in X^{+}_{0}$, we construct the \emph{Kac module} $K(\la)$ and the \emph{dual Kac module} $K^{-}(\la)$
by using the tensor product and the Hom-space in the following way:
$$
K^{+}(\lambda):=K(\lambda):=U({\mathfrak g})\otimes_{U({\mathfrak p}^{+})} L_{0}(\lambda)
\qquad \text{and} \qquad K^{-}(\lambda):=
\Hom_{U({\mathfrak p}^{-})}\left(U({\mathfrak g}), L_{0}(\lambda) \right).
$$

The module $K(\lambda)$ has a unique maximal submodule. The head of $K(\lambda)$ is the simple
finite dimensional ${\mathfrak g}$-module $L(\lambda)$. Then $\{L(\lambda) \mid \lambda \in X_{0}^{+}\}$
is a complete set of non-isomorphic simple modules in $\mathcal{F}_{(\g,\g_{0})}$.  Let $P(\lambda)$
(resp.\ $I(\lambda)$) denote the projective cover 
(resp.\ injective hull) in $\mathcal{F}_{(\g,\g_{0})}$ for the
simple ${\mathfrak g}$-module $L(\lambda)$. These are all finite dimensional. Moreover, the projective
covers admit filtrations with sections being Kac modules and the injective hulls
have filtrations whose sections are dual Kac modules. These filtrations also respect the dominance
ordering on weights and thus ${\mathcal F}_{(\g,\g_{0})}$ is a highest 
weight category (cf.\ \cite[Section 3]{BKN3}) as defined in
\cite{CPS1}.

\subsection{\texorpdfstring{The Lie Superalgebra $\gl (m|n)$}{The Lie Superalgebra gl(m|n)}}\label{SS:glmn}  The standard example of a Type I classical Lie superalgebra  
is $\g=\mathfrak{gl}(m|n)$. As a vector space $\g$ is the set of $m+n$ by $m+n$ matrices, and 
one may take the matrix units $E_{i,j}$ where $1\leq i,j \leq m+n$ as a basis. The even component
${\mathfrak g}_{\0}$ is the span of $E_{i,j}$ where $1\leq i,j  \leq m$ or $m+1 \leq i,j \leq m+n$.   A basis for ${\mathfrak g}_{\1}$ is given by the $E_{i,j}$ such that $m+1\leq i \leq m+n$ and $1\leq j \leq n$ or $1\leq i \leq m$ and $m+1 \leq j \leq m+n$. 
As a Lie algebra ${\mathfrak g}_{\0}\cong \mathfrak{gl}(m)\times \mathfrak{gl}(n)$, and
the corresponding reductive group is $G_{\0}\cong GL(m)\times GL(n)$. Note that $G_{\0}$ acts on
$\mathfrak g_{\1}$ via the adjoint representation. As $\gl (m|n) \cong \gl (n|m)$ and $\gl (m|0) =\gl (m|0)_{\0}$, we may assume without loss that $m \geq n \geq 1$.

Observe that $\fg$ has a $\Z$-grading given by setting $\fg_{0}=\fg_{\0}$ and ${\mathfrak g}_{-1}$ (resp.\ ${\mathfrak g}_{1}$) equal to the lower triangular matrices (resp.\ upper triangular matrices) in ${\mathfrak g}_{\1}$.  In particular, ${\mathfrak g}_{\1}={\mathfrak g}_{-1}\oplus {\mathfrak g}_{1}$.  Furthermore, note that the bracket respects the $\Z$-grading.

%The action of $G_{0}$ on $\g_{-1}$ is given by $(A,B).X=BXA^{-1}$ so the orbits are the matrices of a given rank in $\g_{-1}$. 

We will now establish some basic notation involving the root datum of $\fg$ which will be used later.  Let $\ft$ be the Cartan subalgebra of $\fg$ of all diagonal matrices and let $\fb$ be the Borel subalgebra of all upper triangular matrices. Then $\ft$ is a Cartan subalgebra and $\fb_{0}:=\fb \cap \fg_{0}$ is a Borel subalgebra for $\fg_{0}$.  With respect to these choices we can make the root system and $X_{0}^{+}$ explicit as follows.  For $i=1, \dotsc , m+n$, let $\varepsilon_{i}: \ft \to \C$ be the linear functional which picks out the $i$th diagonal entry. With respect to this basis we define a bilinear form on $\ft$ by 
\[
(\varepsilon_{i}, \varepsilon_{j}) = \begin{cases} \delta_{i,j}, &1\le i,j \le m; \\
                                                          -\delta_{i,j}, &m+1 \le i,j \le m+n; \\
                                                          0, &\text{otherwise}.
\end{cases} 
\]

As with $\gl (m+n)$, the set 
\begin{equation*}
\Phi =\left\{\varepsilon_{i}-\varepsilon_{j} \mid 1 \leq i,j \leq m+n, i \neq j \right\}
\end{equation*} is the set of roots for $\fg$ and 
\[
\Phi^{+}  =\left\{\varepsilon_{i}-\varepsilon_{j} \mid 1 \leq i<j \leq m+n \right\}
\] is the set of positive roots.  The set 
\[
\Phi_{\0} = \left\{ \varepsilon_{i} - \varepsilon_{j} \mid 1 \leq i \neq j \leq m \text{ or } m+1 \leq  i\neq j \leq m+n \right\}
\] is the the set of \emph{even} roots.  The \emph{odd} roots are then $\Phi_{\1} = \Phi \backslash \Phi_{\0}$.  We set $\Phi_{\0 }^{+} = \Phi^{+}\cap \Phi_{\0 }$ and  $\Phi_{\1}^{+} = \Phi^{+}\cap \Phi_{\1}$.
With respect to our choices, 
\begin{equation*}
X^{+}_{0}= \left\{\lambda = \sum_{i=1}^{m+n} \lambda_{i}\varepsilon_{i} \mid \lambda_{i} - \lambda_{i+1} \in \Z_{\geq 0} \text{ for } i\neq m, m+n \right\}.
\end{equation*}  As we discuss below, without loss we can and will assume $\lambda_{i} \in \Z$ for $i=1, \dotsc, m+n$.

Let  
\[
\Phi_{m}^{+} = \left\{\varepsilon_{i}-\varepsilon_{j} \mid 1 \leq i<j \leq m \right\}
\]
be the set of positive roots for the subalgebra of $\fg_{0}$ isomorphic to $\gl(m)$,  Similarly, let 
\[
\Phi_{n}^{+}=\left\{\varepsilon_{i}-\varepsilon_{j} \mid m+1 \leq i<j \leq m+n \right\}
\]
be the set of positive roots for the subalgebra of $\fg_{0}$ isomorphic to $\gl(n)$.

Now define 
\begin{align*}
\rho &= m\varepsilon_{1} + (m-1)\varepsilon_{2}+ \dotsb +\varepsilon_{m} - \varepsilon_{m+1} - 2\varepsilon_{m+2}-\dotsb -n\varepsilon_{m+n},  \\
\rho_{m} &= m\varepsilon_{1} + (m-1)\varepsilon_{2}+ \dotsb +\varepsilon_{m},\\ 
\rho_{n} &= - \varepsilon_{m+1} - 2\varepsilon_{m+2}-\dotsb -n\varepsilon_{m+n}. 
\end{align*}
Then $\rho = \rho_{m}+\rho_{n}$ and the elements $\rho$, $\rho_{m}$, $\rho_{n}$ are each a constant shift of the analogous elements defined via half sums of positive roots.  This shift has no effect on the contents of this paper so we choose to use the more convenient elements defined above.

Given $\lambda \in X_{0}^{+}$, we define the \emph{atypicality} of $\lambda$, $\atyp (\lambda)$, to be the maximal number of pairwise orthogonal elements of $\Phi^{+}$ which are also orthogonal to $\lambda + \rho$.  The atypicality is an integer ranging 
from $0, \dotsc , \min(m,n)$.  If $L(\lambda)$ is a simple $\fg$-module of highest weight $\lambda$, then we define $\atyp \left( L(\lambda)\right) := \atyp (\lambda)$.  It is known that the atypicality of a simple module is independent of the choice of Cartan and Borel subalgebras and, furthermore, is the same for all simple modules in a given block.  Hence it makes sense to refer to the atypicality of a block.

If $\lambda = \sum_{i=1}^{m+n}\lambda_{i}\varepsilon_{i} \in X_{0}^{+}$ has atypicality zero, then by \cite[Theorem 1]{Kacnote} $P(\lambda)=K(\lambda)=L(\lambda)$ and, in particular, $K(\lambda)$ and $L(\lambda)$ have complexity zero, which is 
consistent with \cref{E:Kaccomplexity} and \cref{E:simplecomplexity}.  If $\lambda$ has atypicality greater than zero, then since $\gl (m|n)$ has the one-dimensional Berezinian representation of weight $\varepsilon_{1}+\dotsb +\varepsilon_{m} - \varepsilon_{m+1}- \dotsb - \varepsilon_{m+n}$ we may tensor by a suitable one-dimensional representation (doing so clearly preserves complexity) and assume that that $\lambda_{1}, \dotsc , \lambda_{m+n}$ are integers.  Therefore, without loss we always assume $\atyp (\lambda)\geq 1$ and elements of $X_{0}^{+}$ have integral coefficients.

\subsection{}\label{SS:bruhat}  Let ${\mathcal F}=\F_{(\gl (m|n), \gl (m|n)_{0})}$. Serganova provided a convenient combinatorial description of the blocks of ${\mathcal F}$ which we now recall.  
Given a simple $\fg$-module $L(\lambda)$ with $\atyp (\lambda) =k$, there exists a (unique) set of $k$ positive odd pairwise orthogonal roots 
\begin{equation}\label{E:placesofatypicality}
\Omega := \{\varepsilon_{i_{t}}-\varepsilon_{j_{t}} \mid t =1, \dotsc , k \}
\end{equation}
such that $(\la + \rho, \varepsilon_{i_{t}}-\varepsilon_{j_{t}})=0$ for all $t = 1, \dotsc , k$ and where $1 \leq i_{1}, \dotsc , i_{k} \leq m$ and $m+1 \leq j_{1}, \dotsc , j_{k} \leq m+n$.  The \emph{core} of $\lambda$ is the pair of multisets
\begin{multline}
\left(\ \{(\lambda +\rho , \varepsilon_{s}) \mid s \in \{1, \dotsc , m \} \backslash \{i_{1}, \dotsc , i_{k} \}  \}, \right. \\  
\left. \{(\lambda +\rho , \varepsilon_{s}) \mid s \in \{m+1, \dotsc , m+n \} \backslash \{j_{1}, \dotsc , j_{k} \}  \}\ \right).
\end{multline} 
We then have the following description of the blocks of $\F$.

\begin{prop}[\cite{serganova3}]\label{P:serganovasblockdescription}  If $L(\lambda)$ and $L(\mu)$ are two simple modules in $\mathcal{F}$ then $L(\lambda)$ and $L(\mu)$ lie in the same block if and only if $atyp(\lambda) = \atyp (\mu)$ and the core of $\lambda$ equals the core of $\mu$.
\end{prop}

Given a block $\B $ of $\F$, we will abuse notation slightly by writing $\lambda \in \B $ to mean that the simple module $L(\lambda)$ lies in the block $\B$.  For example, we write $\B _{0}$ for the principal block of $\F$ and so, by definition, $0 \in \B _{0}$.

Given $\lambda = \sum_{i=1}^{m+n}\lambda_{i}\varepsilon_{i} \in X^{+}_{0}$ with $\atyp (\lambda) = k$ we set $\lambda^{+}=\sum_{i=1}^{m}\lambda_{i}\varepsilon_{i}$ and define two length functions as follows.  The ``naive'' length function is given by 
\[
|\lambda| = \lambda_{1}+\dotsb +\lambda_{m}.
\] The other length function is given by 
\[
l(\lambda) = k(k+1)/2+\sum_{\alpha \in \Omega} (\lambda^{+}+\rho_{n}, \alpha).
\]  By \cite[Remark 3.4]{suzhang} the function $l$ defines a length function in the sense of \cite{brundan1}. 
Note that if $\lambda \in \B _{0}$ for $\gl (k|k)$, then $l(\lambda) = |\lambda|$. 

For $\lambda,\mu \in \B $, we write $\lambda \leq \mu$ if $\mu - \lambda$ is a sum of positive roots (i.e., the usual dominance order).  We write $\lambda \preccurlyeq \mu$ for the Bruhat order of \cite{brundan1}.  Note that if $\lambda = \sum \lambda_{i}\varepsilon_{i}, \mu = \sum \mu_{i}\varepsilon_{i} \in \B _{0}$ for $\gl (k|k)$, then $\lambda \preccurlyeq \mu$ if and only if $\lambda_{i} \leq \mu_{i}$ for $i=1, \dotsc, m$.\footnote{Note that the bilinear form used by Brundan is the negative of the one used here.}

Given a highest weight $\lambda$ with $\atyp (\lambda) = k$, let $i_{1}, \dotsc , i_{k}, j_{1}, \dotsc , j_{k}$ be as in \cref{E:placesofatypicality}. We partition the elements of $\Phi_{m}^{+}$ into three sets as follows.  Let
\begin{align*}
A_{m} &= \left\{\alpha= \varepsilon_{s}-\varepsilon_{t} \in \Phi_{m}^{+} \mid |\{s,t \}\cap \{i_{1}, \dotsc , i_{k} \}| = 0 \right\}, \\
B_{m} &= \left\{\alpha= \varepsilon_{s}-\varepsilon_{t} \in \Phi_{m}^{+} \mid  |\{s,t \}\cap \{i_{1}, \dotsc , i_{k} \}| = 1 \right\}, \\
C_{m} &= \left\{\alpha= \varepsilon_{s}-\varepsilon_{t} \in \Phi_{m}^{+} \mid |\{s,t \}\cap \{i_{1}, \dotsc , i_{k} \}| = 2 \right\}. 
\end{align*} These sets obviously depend on $\lambda$.  When appropriate we write $A_{m}(\lambda)$, $B_{m}(\lambda)$, etc.\  to remind the reader of these dependencies.  Observe that $\Phi_{m}^{+} = A_{m}\sqcup B_{m} \sqcup C_{m}$.
Define $A_{n}, B_{n}, C_{n}$ analogously by replacing $\Phi_{m}^{+}$ with $\Phi_{n}^{+}$ and $\{i_{1}, \dotsc , i_{k} \}$ with $\{j_{1}, \dotsc , j_{k} \}$.

\subsection{}\label{SS:geometry}  As it will be needed in what follows, we briefly review the geometric structure of $\fg_{1}$ for $\gl (m|n)$.  Note that $G_{0}$ acts on the variety $\fg_{1}$ via the adjoint action.   Namely, in the matrix realization of $\fg$ given in \cref{SS:glmn} 
the action of $G_{0}\cong GL(m)\times GL(n)$ on $\fg_{1}$ is given by $(A,B)\cdot x = AxB^{-1}$ for $A\in GL(m),\ B\in GL(n),\ x\in \fg_{1}$. 

The $G_{0}$-orbit structure of $\fg_{1}$ is given as follows.  The orbits are 
\[
(\fg_{1})_{r} = \{\, x\in\fg_{1}\mid \rk(x)=r\,\}
\]
for
$0\le r\le \min(m,n)$ and, in particular, we have
\[
(\fg_{1})_{r} = G_{0}.x_{r},
\] where $x_{r}$ is any fixed matrix of rank $r$.

 The closure of $(\fg_{1})_{r}$ is
\[
\overline{(\fg_{1})_{r}} = \{\, x\in\fg_{1}\mid \rk(x)\le r\,\};
\]
thus $(\fg_{1})_{r} \subset \overline{(\fg_{1})_{s}}$ if and only if $r \le s$.  Hence, the graph (Hasse diagram) which describes the partial ordering given by inclusion of orbit closures is a simple chain.  

For $M  \in \F$, let $\V_{\fg_{1}}(M)$ denote the support variety of $M$ as defined in \cref{SS:g1supports}.  In particular, $\V _{\fg_{1}}(\C ) = \fg_{1}$.  It follows, since $\V_{\fg_{1}}(M)$ is a closed $G_{0}$-invariant subvariety of $\V_{\fg_{1}}(\C)$, that
$\V_{{\mathfrak g}_{1}}(M)=\overline{(\fg_{1})_{r}}$ for some $r$. 
So $\V_{{\mathfrak g}_{1}}(M)$ is always irreducible and can be computed by applying the rank variety description to a representative from each of the $\min(m,n)$ orbits. Note that a similar 
description of $G_{0}$-orbits on ${\mathfrak g}_{-1}$ also holds.

\subsection{Example}\label{SS:gl11example}  Let ${\mathfrak g}=\mathfrak{gl}(1|1)$ and ${\mathcal F}={\mathcal F}_{(\g,\g_{0})}$. 
The simple modules in the principal block ${\mathcal B}_{0}$ are one dimensional and indexed by $L(\lambda\varepsilon_{1}-\lambda\varepsilon_{2} )$ where $\lambda\in {\mathbb Z}$. 
In \cite{BKN3}, we proved that $c_{\mathcal F}(L(\lambda\varepsilon_{1}-\lambda\varepsilon_{2}))=2$ for all $\lambda\in {\mathbb Z}$ by constructing an explicit minimal projective resolution. However, the relative cohomology 
ring $\HH^{\bullet}({\mathfrak g},{\mathfrak g}_{0};{\mathbb C})$ has Krull dimension one, and so is not large enough to use to construct a support 
variety theory which measures the complexity. 

In this case one can consider the subalgebra $\f\cong \mathfrak{sl}(1|1)$ in $\fg$ as defined in \cref{SS:zcomplexity}. The Krull dimension of 
$\HH^{\bullet}(\ff,\ff_{0};\C)$ is two. In \cref{SS:complexity}, we will show that 
one can define a support variety theory for $\ff$ which measures the complexity of modules for $\gl(1|1)$. 
This construction does not easily generalize to $\gl(m|n)$ when $m,n>1$. It remains an open question as to whether there exists 
a theory of varieties for modules for these classical Lie superalgebras which can be used to compute the rate of growth of projective resolutions.

\section{Support Varieties} \label{S:supportvarieties}

\subsection{} In \cite{BKN1} the authors showed that for a classical Lie superalgebra $\fg$ the relative cohomology ring for the pair $({\mathfrak g},{\mathfrak g}_{\0})$ is 
finitely generated.   We then used this ring to construct a support variety theory for objects of $\F_{(\fg , \fg_{\0})}$. These 
varieties provide an important geometric interpretation for the atypicality of simple ${\mathfrak g}$-modules. In this section, for Type I Lie superalgebras, 
we will prove properties about support varieties for ${\mathfrak g}_{\pm 1}$ (resp.\  the pair $({\mathfrak p}^{\pm},{\mathfrak g}_{0})$) 
which will be used to measure the complexity for  ${\mathfrak g}_{\pm 1}$-modules (resp.\  modules in ${\mathcal F}_{({\mathfrak p}^{\pm},{\mathfrak g}_{0})}$). 
The results in this section will later be used to compute the complexity of Kac modules in ${\mathcal F}_{({\mathfrak g},{\mathfrak g}_{0})}$. 
 
Let $\mathfrak{g}$ be a classical Lie superalgebra,
$R:=\operatorname{H}^{\bullet}(\mathfrak{g}, \mathfrak{g}_{\0};{\mathbb C})$, and $M_{1}, M_{2}$
be in ${\mathcal F}:={\mathcal F}_{(\g,\g_{\0})}$. According to \cite[Theorem 2.5.3]{BKN1},
$\Ext_{\mathcal{F}}^{\bullet}(M_{1},M_{2})$ is a finitely generated $R$-module.
Set $J_{({\mathfrak g},{\mathfrak g}_{\0})}(M_{1},M_{2}) :=
\operatorname{Ann}_{R}(\Ext_{\mathcal{F}}^{\bullet}(M_{1},M_{2}))$
(i.e., the annihilator ideal of this module).  The \emph{relative support variety of the pair $(M_{1},M_{2})$} is
\begin{equation}
\mathcal{V}_{(\mathfrak{g},\mathfrak{g}_{\0})}(M_{1},M_{2}) :=
\operatorname{MaxSpec}(R/J_{({\mathfrak g},{\mathfrak g}_{\0})}(M_{1},M_{2})).
\end{equation}

In the case when $M=M_{1}=M_{2}$, set $J_{(\g,\g_{\0})}(M)=J_{(\g,\g_{\0})}(M,M)$, and
$$\mathcal{V}_{(\g,\g_{\0})}(M):=\mathcal{V}_{(\g,\g_{\0})}(M,M).$$
The variety $\mathcal{V}_{(\mathfrak{g},\mathfrak{g}_{\0})}(M)$ is called the \emph{support variety} of $M$.
In this situation, $J_{(\g,\g_{\0})}(M)=\operatorname{Ann}_{R}\left(\operatorname{Id} \right)$ where $\operatorname{Id}$ is the identity morphism in $\Hom_{\mathcal F}(M,M)$.

\subsection{\texorpdfstring{The Case ${\mathfrak g}={\mathfrak g}_{\pm 1}$}{The Case g = g\_1 or g\_{-1}}}\label{SS:g1supports} Observe that for Type I Lie superalgebras both $\fg_{1}$ and $\fg_{-1}$ are abelian Lie superalgebras and, consequently, 
\[
R_{\pm} :=\HH^{\bullet} (\fg_{\pm 1}, \C ) = \HH^{\bullet}(\fg_{\pm 1}, \{0 \};\C) \cong S(\fg_{\pm 1}^{*})
\] as graded algebras.  Let $\mathcal{F}(\fg_{\pm 1})$ be the category of finite dimensional $\fg_{\pm 1}$-modules.  If $M$ is an object in $\mathcal{F}(\fg_{\pm 1})$, then one can define the $\fg_{\pm 1}$ \emph{support variety} of $M$, 
%just as in \cite{BKN1}.  Namely, set
%\[
%I_{M} = \left\{r \in R_{\pm} \mid r.m=0 \text{ for all } m \in \Ext_{\mathcal{F}(\fg_{\pm 1})}^{\bullet}(M,M) \right\}
%\] and then the support variety of $M$ is 
\[
\V_{\fg_{\pm 1}}(M) = \V_{(\fg_{\pm 1},0)}(M).
%\operatorname{MaxSpec}\left(R_{\pm}/I_{M} \right).
\]  Since $\fg_{\pm 1}$ is abelian the arguments given in \cite[Section 5]{BKN1} for detecting subalgebras apply here as well and one has that $\V_{\fg_{\pm 1}}(M)$ is canonically isomorphic to the following rank variety: 
\[
\V_{\fg_{\pm 1}}^{\text{rank}}(M) :=\left\{ x \in \fg_{\pm 1} \mid M \text{ is not projective as a $U(\langle x \rangle)$-module} \right\} \cup \{ 0 \},
\] where $U(\langle x \rangle)$ denotes the enveloping algebra of the Lie subsuperalgebra generated by $x \in \fg_{\pm 1}$.  We will identify $\V_{\fg_{\pm 1}}(M)$ and $\V^{\text{rank}}_{\fg_{\pm 1}}(M)$ via this canonical isomorphism.  As a consequence of this alternate description $\V_{\fg_{\pm 1}}(M)$ satisfies the various properties of a rank variety (e.g., it satisfies the tensor product rule and detects $\fg_{\pm 1}$ projectivity; cf.\ \cite[Sections 5, 6]{BKN1}). 

\subsection{\texorpdfstring{The Case ${\mathfrak g}={\mathfrak p}^{\pm }$}{The Case g = p+ or p-}} For a Type I classical Lie superalgebra we show  
for modules in $\F=\F_{(\fp^{\pm},\fg_{0})}$ that the theory of support varieties for $\fg_{\pm 1}$, as presented in the previous 
section,  does measure complexity.  Recall that if $L_{0}(\lambda)$ is a finite dimensional simple $\fg_{0}$-module, then it is canonically a simple $\fp^{\pm}$-module via inflation.  Furthermore, as $\lambda$ ranges over $X_{0}^{+}$ this provides a complete, irredundant set of simple $\fp^{\pm}$-modules in $\F_{(\fp^{\pm},\fg_{0})}$. 

\begin{theorem} \label{T:psupports} Let $\fg$ be a Type I classical Lie superalgebra and let $M$ be a module in $\F=\F_{(\fp^{\pm},\fg_{0})}$. Then 
$$c_{\mathcal F}(M)=\dim {\mathcal V}_{{\mathfrak g}_{\pm 1}}(M)=\dim \V_{\fg_{\pm 1}}^{\rm{rank}}(M).$$ 
\end{theorem} 

\begin{proof} Let $M$ be in ${\mathcal F}={\mathcal F}_{({\mathfrak p}^{\pm},{\mathfrak g}_{0})}$. According to %Proposition~
\cref{P:complexityext}, it follows that 
$$
c_{\F}(M)=r\left(\Ext^{\bullet}_{(\fp^{\pm},\g_{0})}\bigg(M,\bigoplus_{\lambda\in X_{0}^{+}}  L_{0}(\lambda)^{\dim P(\lambda)}\bigg)\right). 
$$
In this instance the projective cover $P(\lambda)\cong U({\mathfrak p}^{\pm})\otimes_{U({\mathfrak g}_{0})}L_{0}(\lambda)$ in ${\mathcal F}$.
Set 
\[
L=\bigoplus_{\lambda\in X_{0}^{+}}  L_{0}(\lambda)^{\dim P(\lambda)}.
\]
 Observe that as a $G_{0}$-module, 
\begin{align}\label{E:Lcong}
L &\cong [\oplus_{\lambda\in X_{0}^{+}} L_{0}(\lambda)\otimes L_{0}(\lambda)^{*}]\otimes \Lambda^{\bullet}({\mathfrak g}_{\pm 1}) \notag \\
   & \cong  k[G_{0}]\otimes \Lambda^{\bullet}({\mathfrak g}_{\pm 1}), 
\end{align}
where in the first isomorphism the action of $G_{0}$ on $ \Lambda^{\bullet}({\mathfrak g}_{\pm 1})$ and on each $ L_{0}(\lambda)^{*}$ is trivial, and in the second isomorphism the action on $k[G_{0}]$ is by left translation and the action on $\Lambda^{\bullet}({\mathfrak g}_{\pm 1})$ is trivial.  The second isomorphism is a well known fact from the representations of reductive algebraic groups (cf.\ \cite[I.3.7]{jantzen}).

Next observe that $\fg_{\pm 1}$ is an ideal of $\fp^{\pm }$. There exists a Lyndon-Hochschild-Serre spectral sequence 
for the pair $(\fg_{\pm 1}, \{0 \})$ in $(\fp^{\pm }, \fg_{0})$ (cf.\ \cite[Theorem 6.5]{BorelWallach}): 
$$
E_{2}^{i,j}=\Ext^{i}_{(\fg_{0},\fg_{0})}({\mathbb C}, \Ext^{j}_{(\fg_{\pm 1},\{0\})}(M,L))\Rightarrow   
\Ext^{i+j}_{(\fp^{\pm },\fg_{0})}(M,L).
$$  
The higher extension groups $\Ext_{(\fg_{0},\fg_{0})}^{i}(-,-)$ vanish because we are considering extensions in the category of finitely semisimple
$\fg_{0}$-modules, thus the spectral sequence collapses and yields the first of the following isomorphisms:  
\begin{align*}
\Ext^{d}_{(\fp^{\pm},\g_{0})}(M,L)&\cong \Hom_{\g_{0}}({\mathbb C},\Ext^{d}_{\g_{\pm 1}}(M,L)),\\
&\cong \Hom_{G_{0}}({\mathbb C},\Ext^{d}_{\g_{\pm 1}}(M,{\mathbb C})\otimes L),\\
&= \Hom_{G_{0}}({\mathbb C},\Ext^{d}_{\g_{\pm 1}}(M,{\mathbb C})\otimes k[G_{0}]\otimes \Lambda^{\bullet}({\mathfrak g}_{\pm 1}), ),\\
&\cong  \Hom_{{\mathbb C}}({\mathbb C},\Ext^{d}_{\g_{\pm 1}}(M,{\mathbb C})\otimes \Lambda^{\bullet}({\mathfrak g}_{\pm 1})). 
\end{align*}  The second isomorphism follows from the fact that $\fg_{\pm 1}$ acts trivially on $L$ and the equality follows from \cref{E:Lcong}.  The third isomorphism follows from the Tensor Identity, the fact that $k[G_{0}] \cong \Ind_{1}^{G_{0}} \C $, and Frobenius reciprocity.

Hence, for all $d \geq 0$, we see that
\[
\dim \Ext^{d}_{(\fp^{\pm},\g_{0})}(M,L)\cong \dim \Ext^{d}_{\fg_{\pm 1}}(M,{\mathbb C})\otimes \Lambda^{\bullet}({\mathfrak g}_{\pm 1}),
\] 
and, since $\Lambda^{\bullet}({\mathfrak g}_{\pm 1})$ is finite dimensional, we have the first of the following equalities:
$$c_{\F}(M)=r(\Ext^{\bullet}_{\fg_{\pm 1}}(M,{\mathbb C}))=c_{\F(\fg_{\pm 1})}(M) =\dim {\mathcal V}_{\g_{\pm 1}}(M)=\dim {\mathcal V}_{\g_{\pm 1}}^{\text{rank}}(M).$$
To obtain the subsequent equalities we use  \cite[Theorem 2.9.1]{BKN3} along with the fact that the abelian superalgebra $\fg_{\pm 1}$ has only a single simple module, namely the trivial module. 
\end{proof}

\section{Kazhdan-Lusztig Polynomials}\label{S:KLpolys}

\subsection{} Given $\lambda, \mu \in X^{+}_{0}$, we define the ``naive'' Kazhdan-Lusztig polynomial $p_{\lambda, \mu}(q)$ by 
\begin{equation*}
p_{\lambda,\mu}(q) = q^{l(\mu) - l(\lambda)}\sum_{n \geq 0} \dim \Ext^{n}_{\F}\left(K(\lambda), L(\mu) \right)q^{-n}.
\end{equation*}
In the notation of \cite[Theorem 4.51]{brundan1} we have 
\begin{equation}\label{E:KLBrundanpoly}
p_{\lambda, \mu}(q) = q^{l(\mu) - l(\lambda)}l_{\lambda, \mu}(-q).
\end{equation}
In particular, $p_{\lambda, \mu}(q)$ has constant term $1$.

\subsection{} For the purposes of our computation we will need to use the following fact that 
the set of Kazhdan-Lusztig polynomials is finite. 

\begin{theorem}\label{T:finitenumberofKLpolys}  For a fixed $\gl (m|n)$, the set $\left\{ p_{\lambda,\mu}(q) \mid \lambda, \mu \in X_{0}^{+} \right\}$ is finite.
\end{theorem}

\begin{proof}  If $K(\lambda)$ and $L(\mu)$ lie in different blocks, then $p_{\lambda, \mu}=0$.  So we may assume $\lambda$ and $\mu$ lie in the same block.  Now fix a block $\mathcal{B}$ of atypicality $k$.  In \cite[Section 3.9]{suzhang} Su and Zhang combinatorially define a bijection on highest weights between any block of atypicality $k$ and the principal block of $\gl (k|k)$.  Let $\phi : \B \to \B_{0, k|k}$ denote the Su-Zhang bijection from $\B$ to the principal block of $\gl (k|k)$.  We discuss the bijection in greater detail in \cref{SS:generalblock}.  For the moment, however, we only need the following fact.  By \cite[Theorem 3.29]{suzhang} and the fact that the Su-Zhang bijection satisfies $l(\gamma) = l(\phi(\gamma))$ for all $\gamma \in \B$, we have $p_{\lambda, \mu}(q) = p_{\phi(\lambda), \phi(\mu)}(q)$ for all $\lambda, \mu \in \B$.  As a consequence, 
\[
\left\{ p_{\lambda,\mu}(q) \mid \lambda, \mu \in X_{0}^{+} \right\} = \left\{ p_{\lambda,\mu}(q) \mid \lambda, \mu \in \B_{0, k|k}, k=1, \dotsc , \operatorname{min}(m,n) \right\} \cup \{0 \}.
\] Therefore we may assume without loss of generality that $\B = \B_{0}$ is the principal block of $\gl (k|k)$.

Since by definition the coefficients of the polynomials $p_{\lambda,\mu}(q)$ are nonnegative integers, it suffices to show that there is an absolute bound on their degree and on the sum of their coefficients.
In order to bound the degree, observe that 
\begin{equation}\label{E:KLdegree}
\begin{aligned}
\dim \Ext_{\F}^{d}(K(\lambda), L(\mu)) &= \dim \Ext^{d}_{(\fg, \fg_{0})}(K(\lambda), L(\mu)) \\
                                      &= \dim  \Ext^{d}_{(\fp^{+}, \fg_{0})}(L_{0}(\lambda), L(\mu)) \\
                                      &= \dim  \Hom_{\fg_{0}}(L_{0}(\lambda), H^{d}(\fg_{1}, 0; L(\mu))) \\
                                      &\leq \dim  \Hom_{\fg_{0}}(L_{0}(\lambda), S^{d}(\fg_{1}^{*})\otimes L(\mu)) \\
                                      &\leq \dim \Hom_{\fg_{0}}(L_{0}(\lambda), S^{d}(\fg_{1}^{*})\otimes \Lambda^{\bullet}(\fg_{-1})\otimes L_{0}(\mu)).
\end{aligned}
\end{equation}
The first line is \cref{E:relcohom}, the second is Frobenius reciprocity, the third is application of a Lyndon-Hochschild-Serre spectral sequence, the fourth is because in this case relative cohomology is a subquotient of the complex $S^{\bullet}(\fg_{1}^{*}) \otimes L(\mu)$ (cf.\ \cite[Section 2.5]{BKN1}), and the last is because $L(\mu)$ is a quotient of the Kac module $K(\mu)$ and, hence, any $\fg_{0}$ composition factor of $S^{n}(\fg_{1}^{*}) \otimes L(\mu)$ is a composition factor of $S^{d}(\fg_{1}^{*}) \otimes K(\mu) \cong S^{n}(\fg_{1}^{*}) \otimes \Lambda^{\bullet}(\fg_{-1})\otimes L_{0}(\mu)$ (this isomorphism is as $\fg_{0}$-modules).

Now consider the element $c:=\sum_{k=1}^{m}E_{k,k} \in \fg_{0}$.  Then $c$ is central in the enveloping algebra of $\fg_{0}$. Furthermore, if $L_{0}(\gamma)$ is a simple $\fg_{0}$-module of highest weight $\gamma \in \fh^{*}$, then $c$ acts on $L_{0}(\gamma)$ by the scalar $|\gamma|$.  From this fact and the description of $\fg_{\pm 1}$ as $\fg_{0}$-modules, we see that $L_{0}(\lambda)$ is a composition factor of $S^{d}(\fg_{1}^{*})\otimes \Lambda^{\bullet}(\fg_{-1})\otimes L_{0}(\mu)$ only if 
\begin{equation*} 
|\lambda| = -d - b + |\mu|
\end{equation*}
 for some $b\in \{0,1, \dots , \dim \fg_{-1}\}$.  By \cref{E:KLdegree}, it follows that 
\begin{equation} \label{E:KLdegreebounds}
\dim \Ext_{\F}^{d}(K(\lambda), L(\mu) )\neq 0 \implies -d = |\lambda| - |\mu| + b
\end{equation}
where  $b\in \{0,1, \dots , \dim \fg_{-1}\}$.  This statement along with the fact that $|\lambda| = l(\lambda)$ for all $\lambda$ in the principal block of $\gl (k|k)$ shows by the definition of $p_{\lambda,\mu}(q)$ that the degree of this polynomial is bounded by $\dim \fg_{-1}$.

Next we prove that the sum of the coefficients of $p_{\lambda, \mu}(q)$ is absolutely bounded.  Observe using \cref{E:KLBrundanpoly} that this sum is given by $p_{\lambda, \mu}(1) = l_{\lambda,\mu}(-1)$.  
There is a convenient alternate description of $l_{\lambda,\mu}(-q^{-1})$ given in \cite[Theorem 3.24]{suzhang} which defines it as a sum of monomials indexed by a subset of the symmetric group on $k$ letters.  From this it is evident that $p_{\lambda,\mu}(1) \leq k!$.  
\end{proof}

\section{Projective Modules}\label{S:projmodulebound}

\subsection{} The formula in \cref{P:complexityext} indicates that in order to compute the complexity of a 
module in $\F_{(\fg , \fg_{0})}$, one requires effective bounds on the dimension of $P(\mu)$ (the projective cover of the simple $\gl (m|n)$-module $L(\mu)$).  
In order to accomplish this we relate the dimension of $P(\mu)$ to the dimension of $L_{0}(\mu)$.  

First observe that as a $\fg_{0}$-module $L(\mu)$ contains $L_{0}(\mu)$ as a composition factor.  Consequently, 
$\dim P(\mu) \geq \dim L_{0}(\mu)$.  On the other hand, by the PBW theorem for Lie superalgebras we see that $U(\fg )$ is a free $U(\fg_{0})$-module and, thus, 
$ U(\fg ) \otimes_{U(\fg_{0})} L_{0}(\mu)$ is a projective $\fg$-module.  Furthermore, by applying Frobenius reciprocity we see that $U(\fg ) \otimes_{U(\fg_{0})} L_{0}(\mu)$ surjects onto $L(\mu)$.  Thus, 
$P(\mu)$ is a direct summand of $ U(\fg ) \otimes_{U(\fg_{0})} L_{0}(\mu)$.  However, by the PBW theorem, $ U(\fg ) \otimes_{U(\fg_{0})} L_{0}(\mu) \cong \Lambda^{\bullet}(\fg_{\1}) \otimes L_{0}(\mu)$ as vector spaces.  Therefore we have 
$2^{\dim \fg_{\1}}\dim L_{0}(\mu) \geq \dim P(\mu)$.  In summary we have: 
\begin{equation}\label{E:boundingprojectives}
2^{\dim \fg_{\1}}\dim L_{0}(\mu) \geq \dim P(\mu)  \geq \dim L_{0}(\mu).
\end{equation}

\subsection{} We now obtain an upper bound on the dimension of projective indecomposables appearing in a minimal projective resolution of an object of $\F$. 

\begin{theorem}\label{T:Projupperbound}  Let $M$ be a $\gl (m|n)$-module which lies in a block of atypicality $k$.  Let $P_{\bullet} \to M$ be a minimal projective resolution for $M$.  Then there is a positive constant $C$ depending only on $m$, $n$ and $M$ such that if $P(\mu)$ appears as a direct summand of $P_{d}$, then 
\[
\dim P(\mu) \leq C d^{(m+n-k-1)k}.
\]
\end{theorem}

\begin{proof} We first consider the case when $M=L(\lambda)$ is a simple module.   Let $P(\mu)$ be a direct summand of $P_{d}$ in the minimal projective resolution of $L(\lambda)$.  By tensoring by sufficiently many copies of the one-dimensional Berezinian representation (cf.\ \cref{SS:glmn}) we may assume without loss of generality that $\lambda_{i},  \mu_{i} >0$ for $i=1, \dotsc , m$ and $\lambda_{i},  \mu_{i} < 0$ for $i=m+1, \dotsc , m+n$. Furthermore, by \cref{E:boundingprojectives} it suffices to show that $\dim L_{0}(\mu) $ is bounded above by $C d^{(m+n-k-1)k}$.

Since $\fg_{0} \cong \gl (m) \oplus \gl (n)$, it follows by Weyl's dimension formula (cf.\ \cite{goodman}) that 
\begin{equation}\label{E:Lzerodimension}
\dim L_{0}(\mu)  = \prod_{\alpha \in \Phi_{m}^{+}}\frac{(\mu + \rho_{m}, \alpha)}{(\rho_{m}, \alpha)}  \prod_{\alpha \in \Phi_{n}^{+}}\frac{(\mu + \rho_{n}, \alpha)}{(\rho_{n}, \alpha)}.
\end{equation}

We can decompose the first factor as follows: 
\begin{equation}\label{E:tripleproduct}
 \prod_{\alpha \in \Phi_{m}^{+}}\frac{(\mu + \rho_{m}, \alpha)}{(\rho_{m}, \alpha)}=  \prod_{\alpha \in A_{m}}\frac{(\mu + \rho_{m}, \alpha)}{(\rho_{m}, \alpha)} \prod_{\alpha \in B_{m}}\frac{(\mu + \rho_{m}, \alpha)}{(\rho_{m}, \alpha)} \prod_{\alpha \in C_{m}}\frac{(\mu + \rho_{m}, \alpha)}{(\rho_{m}, \alpha)},
\end{equation} where $A_{m}=A_{m}(\mu)$, $B_{m}=B_{m}(\mu)$, and $C_{m}=C_{m}(\mu)$ are as in \cref{SS:bruhat}.

Consider the factor involving $A_{m}$.  Since $P(\mu)$ lies in the same block as $L(\lambda)$ we have by \cref{P:serganovasblockdescription} that the core of $\mu$ equals the core of $\lambda$.  In particular, we have an equality of multisets 
\begin{equation*}
\{(\mu + \rho, \alpha) \mid \alpha \in A_{m}(\mu) \} = \{(\lambda+\rho, \alpha) \mid \alpha \in A_{m}(\lambda) \}.
\end{equation*}
and, hence, we have an equality of multisets 
\begin{equation*}
\{(\mu + \rho_{m}, \alpha) \mid \alpha \in A_{m}(\mu) \} = \{(\lambda+\rho_{m}, \alpha) \mid \alpha \in A_{m}(\lambda) \}.
\end{equation*}
From this and the fact that $(\rho_{m}, \alpha) \geq 1$ for all $\alpha \in \Phi_{m}^{+}$ we deduce that 
\begin{equation}\label{E:productI}
\prod_{\alpha \in A_{m}(\mu)}\frac{(\mu + \rho_{m}, \alpha)}{(\rho_{m}, \alpha)} \leq \prod_{\alpha \in A_{m}(\mu)} (\mu + \rho_{m}, \alpha) =  \prod_{\alpha \in A_{m}(\lambda)}(\lambda+\rho_{m}, \alpha) =: C_{1}.
\end{equation} In particular, $C_{1}$ is a constant which depends only on $m,n$, and $\lambda$.

Next consider the $B_{m}$ factor in \cref{E:tripleproduct}.   Since $P(\mu)$ appears in $P_{d}$, we have  $\Hom_{\mathcal{F}}(P_{d}, L(\mu))\linebreak[0] \neq 0$.  But as $P_{\bullet}$ is a minimal projective resolution, we have 
\[
\Ext_{\F}^{d}\left(L(\lambda), L(\mu) \right) = \Hom_{\F}(P_{d}, L(\mu) )\neq 0.
\] By \cite[Corollary 4.52]{brundan1} we have
\[
0 \neq \dim \Ext^{d}_{\F }(L(\lambda),L(\mu))  =  \sum_{i+j=d}\ \sum_{\s\in \mathcal{B}} \dim \Ext^{i}_{\F}(K(\s),L(\lambda))\,\dim \Ext^{j}_{\F}(K(\s),L(\mu)),
\] where $\mathcal{B}$ is the block containing $L(\lambda)$ and $L(\mu)$.   Therefore, there must be a $\sigma \in \mathcal{B}$ and $i,j$ with $i+j=d$ such that 
\[
 \dim \Ext^{i}_{\F}(K(\s),L(\lambda)) \neq 0 \quad \text{and} \quad \dim \Ext^{j}_{\F}(K(\s),L(\mu)) \neq 0.
\] Now this implies by \cref{E:KLdegreebounds} that 
\begin{align*}
i&=|\lambda|-|\sigma| - b_{1}, \\
j&=  |\mu| -|\sigma| -b_{2},
\end{align*} where $b_{1}, b_{2} \in \{0, 1, \dotsc , \dim \fg_{-1} \}$.  Taking the difference we obtain 
\[
-d \leq i-j= |\lambda|-|\mu| -b_{1}+b_{2}
\] and we obtain
\[
|\mu| \leq d + D,
\] where $D$ is a constant depending only on $m,n$ and $\lambda$.

Since $\mu_{1}, \dotsc , \mu_{m}\geq 0$ we deduce
\[
\mu_{t} \leq d+D, 
\] for $t=1, \dotsc , m$. Let $\alpha = \varepsilon_{r} - \varepsilon_{s} \in B_{m} (\mu)$.  Then 
\[
(\mu, \alpha)  \leq \mu_{r}+\mu_{s} \leq 2(d+D).
\]  

Let $z = \min\{( (\rho_{m}, \alpha) \mid \alpha \in \Phi_{m}^{+}\} > 0$ and $Z = \max\{ (\rho_{m}, \alpha) \mid \alpha \in \Phi_{m}^{+} \}$. We then have
\begin{equation}\label{E:productII}
\prod_{\alpha \in B_{m}}\frac{(\mu + \rho_{m}, \alpha)}{(\rho_{m}, \alpha)} = \prod_{\alpha \in B_{m}}\frac{(\mu, \alpha) + (\rho_{m}, \alpha)}{(\rho_{m}, \alpha)} \leq \prod_{\alpha \in B_{m}}\frac{2(d+D)+Z}{z} \leq C_{2} d^{(m-k)k},
\end{equation}
where $C_{2}$ is a constant depending only on $D$, $Z$, and $z$ and so, ultimately, only on $m,n$, and $\lambda$.
The last inequality also uses the observation that there are precisely $(m-k)k$ elements in the set $B_{m}$.

We now consider the $C_{m}$ factor.  An identical calculation using that there are precisely $(k^{2}-k)/2$ elements in the set $C_{m}$ yields 
\begin{equation}\label{E:productIII}
 \prod_{\alpha \in C_{m}}\frac{(\mu + \rho_{m}, \alpha)}{(\rho_{m}, \alpha)} \leq C_{3} d^{(k^{2}-k)/2}
\end{equation}
where again $C_{3}$ depends only on $m,n$, and $\lambda$.

Finally, we input \cref{E:productI}, \cref{E:productII}, and \cref{E:productIII} into \cref{E:tripleproduct} and obtain 
\begin{equation}\label{E:pbound}
 \prod_{\alpha \in \Phi_{m}^{+}}\frac{(\mu + \rho_{m}, \alpha)}{(\rho_{m}, \alpha)} \leq C_{1}C_{2}C_{3}d^{(m-k)k+(k^{2}-k)/2}.
\end{equation}

We now turn to the second factor in \cref{E:Lzerodimension}.  An identical analysis (changing $m$ to $n$) yields
\begin{equation}\label{E:qbound}
 \prod_{\alpha \in \Phi_{n}^{+}}\frac{(\mu + \rho_{n}, \alpha)}{(\rho_{n}, \alpha)} \leq  C'_{1}C'_{2}C'_{3}d^{(n-k)k +(k^{2}-k)/2}.
\end{equation}  Finally, inserting \cref{E:pbound,E:qbound} into \cref{E:Lzerodimension} we obtain 
\[
\dim L_{0}(\mu) \leq  Cd^{(m+n-k-1)k},
\] where $C$ is some constant depending only on $m,n$, and $\lambda$.  As we explained at the beginning of the proof, this suffices to prove the desired result for $L(\lambda)$.

To prove the general case, we use the Horseshoe Lemma to argue by induction on the length of a composition series for $M$. 
\end{proof}

\section{Complexity for Kac Modules} \label{S:KacComplexity}

\subsection{}  We begin by establishing general bounds using the geometry of support varieties for the complexity of a Kac module for a Type I Lie superalgebra $\fg$.  
Let $K(\lambda)$ be a Kac module and let 
$$\dots \rightarrow P_{2} \rightarrow P_{1} \rightarrow P_{0} \rightarrow L_{0}(\lambda) \rightarrow 0$$ 
be a minimal projective resolution of $L_{0}(\lambda)$ in ${\mathcal F}_{({\mathfrak p}^{+},{\mathfrak g}_{0})}$. 
We can apply the exact functor $U({\mathfrak g})\otimes_{U({\mathfrak p}^{+})} -$ to this resolution to get 
a projective resolution with the same rate of growth for $K(\lambda)$ in ${\mathcal F}_{(\g,\g_{0})}$: 
$$\dots \rightarrow U({\mathfrak g})\otimes_{U({\mathfrak p}^{+})} P_{2} \rightarrow U({\mathfrak g})\otimes_{U({\mathfrak p}^{+})} 
P_{1} \rightarrow U({\mathfrak g})\otimes_{U({\mathfrak p}^{+})} P_{0} \rightarrow K(\lambda) \rightarrow 0.$$ 
This shows that 
\begin{equation} \label{E:upperboundKac}
c_{{\mathcal F}_{(\g,\g_{0})}}(K(\lambda))\leq c_{{\mathcal F}_{({\mathfrak p}^{+},\g_{0})}}(L_{0}(\lambda))\leq \dim \g_{1},
\end{equation} 
where the last inequality is by \cref{T:psupports}.

Next observe that any projective resolution in ${\mathcal F}_{(\g,\g_{0})}$ of a module $M$ (such as $K(\la)$) will restrict to a projective resolution of 
$M$ in ${\mathcal F}_{({\mathfrak p}^{+},{\mathfrak g}_{0})}$. Therefore, 
\begin{equation*}
 c_{{\mathcal F}_{({\mathfrak p}^{+},\g_{0})}}(M) \leq c_{{\mathcal F}_{(\g,\g_{0})}}(M). 
\end{equation*}
Combining this statement with \cref{T:psupports}, we have 
\begin{equation} \label{E:lowerboundKac}
 \dim {\mathcal V}^{\text{rank}}_{\g_{1}}(M) \leq c_{{\mathcal F}_{(\g,\g_{0})}}(M). 
\end{equation} 

\subsection{\texorpdfstring{Complexity of Kac Modules in the Principal Block for $\gl (k|k)$}{Complexity of Kac Modules in the Principal Block for gl(k|k)}}\label{SS:Kacmodulesintheprincipalblock} 
As a step towards solving the general problem for $\gl (m|n)$, we can now compute the complexity of the Kac and dual Kac modules in 
the principal block ${\mathcal B}_{0}$ of ${\mathcal F}_{(\g,\g_{0})}$ for $\mathfrak{gl}(k|k)$. 

\begin{theorem}\label{T:KacinthePrincipalBlock} Let $K(\lambda)$ be a Kac module (resp.\ $K^{-}(\lambda)$ be a dual Kac module) in the principal block $\B_{0}$ of ${\mathcal F}={\mathcal F}_{(\g,\g_{0})}$ for $\g=\mathfrak{gl}(k|k)$. Then
$$c_{\mathcal F}(K(\lambda))=c_{\mathcal F}(K^{-}(\lambda))=\operatorname{atyp}(\lambda)^{2}.$$ 
\end{theorem} 

\begin{proof}  Let us first consider the case of Kac modules. If $K(\lambda)$ lies in the principal block $\B_{0}$, then $\atyp (\lambda)=k$.  According to  \cref{E:upperboundKac} and \cref{E:lowerboundKac}, 
$$\dim {\mathcal V}^{\text{rank}}_{\g_{1}}(K(\lambda))\leq c_{{\mathcal F}_{(\g,\g_{0})}}(K(\lambda))\leq \dim \g_{1}=k^{2}.$$ 
Therefore, it suffices to prove that ${\mathcal V}^{\text{rank}}_{\g_{1}}(K(\lambda))={\mathfrak g}_{1}$. 

Let
\[
I_{k}=E_{1,k+1}+E_{2,k+2}+\dots+E_{k,2k} \in \fg_{1}.
\] The element $I_{k}$ has rank $k$ so by \cref{SS:geometry},    
${\mathfrak g}_{1}=\overline{G_{0}.I_{k}}.$ 
The variety ${\mathcal V}^{\text{rank}}_{\g_{1}}(K(\lambda))$ is closed and stable under $G_{0}$ (because $K(\lambda)$ is a ${\mathfrak g}$-module), so 
we  need only demonstrate that $K(\lambda)$ is not free as $U(\langle I_{k} \rangle)$-module. 

Since $K(\lambda)$ is in ${\mathcal B}_{0}$, $L_{0}(\lambda)\cong S\boxtimes  S^{*}$ where $S$ is a simple $GL(k)$-module. Let $\Delta G_{0}$ (resp.\ 
$\Delta{\mathfrak g}_{0}$) be the image of the diagonal embedding of $GL(k)\hookrightarrow GL(k) \times GL(k)$ (resp.\ $\mathfrak{gl}(k) 
\hookrightarrow \mathfrak{gl}(k)\times \mathfrak{gl}(k)$). As a $\Delta G_{0}$-module,  
$L_{0}(\lambda)\cong S\otimes S^{*}\cong {\mathbb C}\oplus N$ for some module $N$. 

One can verify directly that 
\begin{equation*}
[{\mathfrak g}_{-1},I_{k}]\subseteq \Delta{\mathfrak g}_{0}.
\end{equation*}
By the fact that ${\mathfrak g}$ is ${\mathbb Z}$-graded
\begin{equation*} 
[\Delta{\mathfrak g}_{0},{\mathfrak g}_{-1}]\subseteq {\mathfrak g}_{-1}.
\end{equation*}
The inclusions above imply that as a  $U(\langle I_{k} \rangle)$-module, $K(\lambda)$ decomposes as 
\begin{equation} \label{E:I-decomposition}
K(\lambda)|_{U(\langle I_{k} \rangle)}\cong \left(  U({\mathfrak g})\otimes_{U({\mathfrak p}^{+})}  \C\right)\oplus \left( U({\mathfrak g})\otimes_{U({\mathfrak p}^{+})}  N \right). 
\end{equation}  
Now consider $K(0)=U({\mathfrak g})\otimes_{U({\mathfrak p^{+}})}{\mathbb C}$ and observe that this is the first component of the 
decomposition in \cref{E:I-decomposition} as a $U(\langle I_{k} \rangle)$-module. 
Using the ${\mathbb Z}$-grading on ${\mathfrak g}$ one has 
\begin{equation*} 
[{\mathfrak p^{+}},{\mathfrak g}_{-1}]\subseteq {\mathfrak g}_{-1}\oplus {\mathfrak g}_{0} 
\end{equation*}
and 
\begin{equation*} 
[{\mathfrak g}_{-1},{\mathfrak g}_{0}]\subseteq {\mathfrak g}_{-1}.
\end{equation*}
These two relations imply that as $U({\mathfrak p}^{+})$-module we have 
$$K(0)|_{U({\mathfrak p}^{+})}\cong  \left(1\otimes \C \right)  \oplus  \left( U({\mathfrak g}_{-1}){\mathfrak g}_{-1} \otimes \C \right)\cong  \C \oplus  \left( U({\mathfrak g}_{-1}){\mathfrak g}_{-1} \otimes \C\right).$$ 
Since $I_{k}\in {\mathfrak p}^{+}$, it follows that $K(\lambda)$ as a $U(\langle I_{k} \rangle)$-module has ${\mathbb C}$ as a direct summand, which proves that 
$K(\lambda)$ is not free as a $U(\langle I_{k} \rangle)$-module. 

The proof for dual Kac modules follows the same line of reasoning by interchanging ${\mathfrak g}_{\pm 1}$ with ${\mathfrak g}_{\mp 1}$.
\end{proof}

\subsection{General Case}\label{SS:serganovablockequivalence} We can now compute the complexity of an arbitrary Kac module for $\gl (m|n)$.  This requires the use of Serganova's block equivalences.  Specifically, Serganova proves that a block of $\F =\F (\gl (m|n), \gl (m|n)_{0})$ 
of atypicality $k$ is equivalent to the principal block of $\F (\gl (k|k), \gl (k|k)_{0})$.   This is also proven by entirely different means by Brundan and Stroppel \cite{BrSt:09a}.  However, for our purposes we require the explicit equivalence constructed by Serganova.  We only sketch what we need; the full details can be found in \cite{serganova3}.

Let $\B$ be a fixed block of $\F$ with atypicality $k$.  On a given module in $\B $ the functor defining Serganova's equivalence is given by 
\begin{equation}\label{E:Phidef}
\Res_{\mu}  \circ T_{t} \circ \dotsb \circ T_{1}, 
\end{equation} where $t$ is some natural number determined by the module and the functors $T_{i}$ are certain translation functors which provide  equivalences between blocks of $\F$.   Let $\fg '$ denote the subalgebra of $\fg$ isomorphic to $\gl (k|k)$ spanned by the matrix units $E_{i,j}$  ($i,j = m-k+1, \dotsc , m+k$). The functor $\Res_{\mu}$ refines the restriction from $\fg$ to $\fg '$.

We examine how complexity is affected by the translation functors.  Let $T$ be a translation functor defined by tensoring by a finite dimensional $\fg$-module $E$ and projecting onto a block.   Let $P_{\bullet} \to  M$ be a minimal projective resolution of $M$.  As translation functors are exact, if we apply $T$ to this resolution we obtain a (not necessarily minimal) projective resolution  $T(P_{\bullet}) \to T(M)$.  Furthermore, we have $\dim \left(T(P_{d}) \right) \leq \dim (P_{d}) \cdot \dim (E)$ for all $d \geq 0$.    Therefore we have 
\[
c_{\mathcal{F}}(T(M)) \leq c_{\mathcal{F}}(M).
\]   However, in the case of Serganova's construction, each $T_{i}$ is an equivalence of categories between two blocks with the inverse functor also given by a translation functor.  From this we immediately obtain 
\[
c_{\mathcal{F}}(T_{i}(M)) = c_{\mathcal{F}}(M)
\]  for all $i$ in \cref{E:Phidef}.
  
In what follows we use the notation of \cref{SS:geometry}.  For the computation of the dimension in the following theorem we use the well known formula for the dimension of the variety $\overline{(\fg_{1})_{k}}$ (cf.\ \cite[Proposition 12.2]{Har:92}).

\begin{theorem}\label{T:arbKacmodule}  Let $K(\lambda)$ be a Kac module (resp.\ $K^{-}(\lambda)$ be a dual Kac module) for $\gl(m|n)$ with $\atyp (\lambda) = k$.  Then 
\begin{itemize}
\item[(a)] $c_{\mathcal F}(K(\lambda))= \dim  \overline{(\fg_{1})_{k}} = (m+n)k-k^{2}$; 
\item[(b)] $c_{\mathcal F}(K^{-}(\lambda))= \dim  \overline{(\fg_{-1})_{k}} = (m+n)k-k^{2}$.
\end{itemize} 
\end{theorem}

\begin{proof}  We will prove (a). Part (b) follows by a similar reasoning.  Our first step entails showing that $c_{\mathcal F}(K(\lambda))\geq \dim   \overline{(\fg_{1})_{k}} $.  
As $\atyp (\lambda) =k$, $K(\lambda)$ lies in a block of atypicality $k$. By the work of Serganova \cite{serganova3} discussed above, this block is equivalent to the principal block of $\mathfrak{gl}(k|k)$.  
On a fixed module the functor defining this equivalence is of the form given by \cref{E:Phidef}.  As discussed above, the complexity of a module is unaffected by the functors $T_{1}, \dotsc , T_{t}$.  
Thus the module $M = T_{t} \circ \dotsb \circ T_{1}(K(\lambda))$  has the same complexity as $K(\lambda)$.  By the definition of the functor $\Res_{\mu}$ we have, as $\fg'$-modules, 
\begin{equation}\label{E:Resdecomp}
M = \Res_{\mu}(M) \oplus G_{\mu}(M)
\end{equation}
for a $\fg'$-module $G_{\mu}(M)$ (cf.\ \cite[(4.7.4)]{BKN2}).  The functor \cref{E:Phidef} takes Kac modules to Kac modules and, in particular, $\Res_{\mu}(M)$ is a Kac module in the principal block of $\gl (k|k)$.  Now it was shown in the proof of \cref{T:KacinthePrincipalBlock} that the element $I_{k} \in \fg '_{1}\subseteq \fg_{1}$ lies in $\V_{\fg'_{1}}(\Res_{\mu}(M))$.  By the rank variety description it immediately follows from \cref{E:Resdecomp} that $I_{k} \in \V_{\fg_{1}}(M)$.  But since $I_{k}$ is a rank $k$ matrix and since $ \V_{\fg_{1}}(M)$ is a $G_{0}$-stable closed subvariety of $\fg_{1}$, we have 
\begin{equation}\label{E:Mvariety}
\overline{G_{0}.I_{k}}=\overline{(\fg_{1})_{k}} \subseteq \V_{\fg_{1}}(M).
\end{equation}
Combining this with \cref{E:lowerboundKac} we obtain the desired inequality
\[
c_{\mathcal{F}}(K(\lambda))=c_{\mathcal{F}}(M) \geq \dim \V_{\fg_{1}}(M) \geq \dim \overline{(\fg_{1})_{k}}.
\]

We now consider the reverse inequality.  By \cref{P:complexityext} we have
\begin{equation}\label{E:kaccomplexityext}
c_{\F}(K(\lambda)) = r\left(\Ext^{\bullet}_{(\g,\g_{0})}\left( K(\lambda),\bigoplus  L(\mu)^{\,\dim P(\mu)}\right)\right), 
\end{equation} where the direct sum is over all simple modules in the block which contains $K(\lambda)$.  
%Thus it suffices to show that 
%\[
%\dim \Ext^{d}_{(\g,\g_{0})}\left(K(\lambda),\bigoplus  S^{\,\dim P(\mu)} \right) \leq Cd^{(m+n)k-k^{2}-1},
%\] where $C$ is a constant independent of $d$.  
However, for fixed $d$, 
\begin{equation}\label{E:kacextn}
\dim \Ext^{d}_{(\g,\g_{0})}(K(\lambda),\bigoplus  L(\mu)^{\,\dim P(\mu)}) = \sum \dim P(\mu) \cdot \dim \Ext^{d}_{(\g,\g_{0})}(K(\lambda),  L(\mu)).
\end{equation}  If $P_{\bullet} \to K(\lambda)$ is a minimal projective resolution, then $\Ext^{d}_{(\g,\g_{0})}(K(\lambda),  L(\mu)) = \Hom_{\F}( P_{d},\linebreak[0] L(\mu))$ being nonzero implies $P(\mu)$ is a summand of $P_{d}$.  By \cref{T:Projupperbound} this implies that $\dim P(\mu) \leq Cd^{(m+n-k-1)k}$ for some constant $C$ which depends only on $m,n$, and $\lambda$.  Thus,   
\begin{equation}\label{E:kacextn2}
\dim \Ext^{d}_{(\g,\g_{0})}(K(\lambda),\bigoplus  L(\mu)^{\,\dim P(\mu)}) \leq    Cd^{(m+n-k-1)k}  \dim  \Ext^{d}_{(\g,\g_{0})}(K(\lambda),\bigoplus  L(\mu)).
\end{equation}  Therefore, it suffices to prove 
\begin{equation}\label{E:Kacboundreduction}
\dim \Ext^{d}_{(\g,\g_{0})}(K(\lambda),\bigoplus  L(\mu)) \leq Kd^{k-1}
\end{equation}
for some constant $K$, where the direct sum is over all simple modules in the block containing $K(\lambda)$.

Because the left hand side of \cref{E:Kacboundreduction} is invariant under Serganova's categorical equivalence between blocks, we may assume without loss of generality that $m=n=k$ and 
that $K(\lambda)$ is a Kac module in the principal block for $\gl (k|k)$.  We now consider when 
\begin{equation}\label{E:nonzeroext}
\dim \Ext^{d}_{(\g,\g_{0})}(K(\lambda),L(\sigma)) \neq 0.
\end{equation}
Put $\lambda = \sum_{i=1}^{2k}\lambda_{i}\varepsilon_{i}$ and $\sigma= \sum_{i=1}^{2k} \sigma_{i}\varepsilon_{i}$. As \cref{E:nonzeroext} is nonzero for only finitely many $\sigma$, by tensoring with sufficently many copies of the one-dimensional Berezinian representation we may assume without loss of generality that $(\lambda_{1}, \dotsc , \lambda_{k})$ and $(\sigma_{1}, \dotsc , \sigma_{k})$ are always partitions.  By \cref{E:KLdegreebounds}, the nonvanishing of \cref{E:nonzeroext}  implies that $|\sigma| = |\lambda| + d +b$ where $b \in \{  0, \dotsc , \dim \fg_{-1} =k^{2} \}$.  Therefore, 
\[
|\lambda| + d \leq |\sigma| \leq |\lambda| + d + k^{2}.
\]  That is, in order for \cref{E:nonzeroext}  not to vanish, $(\sigma_{1}, \dotsc , \sigma_{k})$ must be a partition of an integer between $|\lambda| + d$ and $|\lambda| + d+k^{2}$.   

Now the number of partitions of $i$ into not more than $k$ parts is bounded by $C_{1}i^{k-1}$, where $C_{1}$ is a constant depending only on $k$ \cite[Corollary 15.1]{nathanson}.  Furthermore, the dimension of $ \Ext^{d}_{(\g,\g_{0})}(K(\lambda),L(\sigma))$ is the coefficient of a Kazhdan-Lusztig polynomial of the type considered in \cref{S:KLpolys}. By \cref{T:finitenumberofKLpolys} the coefficients of these polynomials are uniformly bounded by some constant $C_{2}$.  Taken together, these observations imply  
\begin{align*}
\dim \Ext^{d}_{(\g,\g_{0})}(K(\lambda),\bigoplus  L(\mu)) &= \sum \dim \Ext^{d}_{(\g,\g_{0})}(K(\lambda), L(\mu)) \\
   & \leq C_{1}\cdot C_{2} \cdot \sum_{i=|\lambda| + d}^{ |\lambda| + d + k^{2}} i^{k-1} \\
   & \leq K d^{k-1},
\end{align*} where $K$ is some constant depending only on $m$, $n$, $k$ and $\la$.  This verifies \cref{E:Kacboundreduction}, and hence completes the proof.
\end{proof}

\subsection{}\label{SS:dufloserganova} Let 
\[
\X = \left\{x \in \fg_{\1} \mid [x,x] = 0 \right\}.
\]  If $M$ is in $\F_{(\fg , \fg_{0})}$, then Duflo and Serganova \cite{dufloserganova} define\footnote{The definition given here is different from but equivalent to the one originally given by Duflo and Serganova (cf.\ \cite[Section 3.6]{BKN2}).} an associated variety by 
\[
\X_{M} = \left\{x \in \X  \mid  M \text{ is not projective as a $U(\langle x \rangle)$-module} \right\} \cup \{ 0 \}
\] and show their varieties capture a number of interesting features of $\F$. 

In the next result we compute the Duflo-Serganova varieties for the Kac and dual Kac modules.  We also indicate how the varieties $\X_{K^{\pm}(\lambda)}$ and $\V_{(\fg , \fg_{0})}(K^{\pm}(\lambda))$ measure the complexity of $K^{\pm}(\lambda)$ in $\F$.

\begin{theorem}\label{T:Kacvarieties}  Let $K^{\pm}(\lambda)$ be a Kac (resp.\ dual Kac) module for $\gl (m|n)$ with $\atyp (\lambda) = k$.  Then 
\begin{itemize} 
\item[(a)] $\X_{K^{\pm}(\lambda)}=\V_{\fg_{\pm 1}}(K^{\pm}(\lambda)) = \overline{(\fg_{\pm 1})_{k}}$;
\item[(b)] $c_{\F}\left(K^{\pm}(\lambda) \right)= \dim \X_{K^{\pm}(\lambda)} + \dim \V_{(\fg , \fg_{0})}\left(K^{\pm}(\lambda) \right)$.
\end{itemize} 
\end{theorem}
\begin{proof}  We will restrict our attention to Kac modules in the proof; the arguments for dual Kac modules are similar. 

(a) We first prove the second equality.  
Since the $\fg_{1}$-support varieties have the tensor product property, it follows by an argument similar to the one given for complexity in \cref{SS:serganovablockequivalence} that 
\[
\V_{\fg_{1}}(K(\lambda)) = \V_{\fg_{1}}(M),
\] where $M=T_{t} \circ \dotsb \circ T_{1}(K(\lambda))$ is the $\gl (m|n)$ module given in the proof of \cref{T:arbKacmodule}.   This along with \cref{E:Mvariety} implies we have 
\[
\overline{(\fg_{1})_{k}} \subseteq \V_{\fg_{1}}(K(\lambda)).
\] On the other hand, by \cref{E:lowerboundKac}  and \cref{T:arbKacmodule} we have that 
\[
\dim \V_{\fg_{1}} \left( K(\lambda) \right) \leq c_{\F (\fg, \fg_{0})} \left(K(\lambda)  \right) = \dim \overline{(\fg_{1})_{k}}.
\] Now since $\dim \V_{\fg_{1}} \left( K(\lambda) \right)$ is a closed $G_{0}$-stable subvariety of $\fg_{1}$, the description of the $G_{0}$-orbits given in \cref{SS:geometry} implies the two varieties are equal. 

We now consider the first equality.  By definition we have 
\[
\V_{\fg_{1}}(K(\lambda)) =  \X_{K(\lambda)} \cap \fg_{1} \subseteq \X_{K(\lambda)}.
\] To prove that the inclusion is in fact an equality we argue by contradiction.  Say $y \in  \X_{K(\lambda)}$ but $y \notin \fg_{1}$.  As $\fg_{\1}= \fg_{-1} \oplus \fg_{1}$, we can write $y= y_{-1} + y_{1}$, with $y_{k} \in \fg_{k}$ for $k=-1,1$ and, by assumption, $y_{-1} \neq 0$.  

 Fix $a \in \R$ with $0<a<1$.  Since $\fg$ is $\Z$-graded we have that $U(\fg )$ is $\Z$-graded and, in turn, $K(\lambda)$ inherits a $\Z$-grading.   We can then define an action of $\Z$ (written multiplicatively with fixed generator $t$) on $\fg$ (resp.\ $K(\lambda)$) by $t.x = a^{l}x$ for $x \in \fg_{l}$ (resp.\ $t.m = a^{l}m$ for $m \in K(\lambda)_{l}$), where $l \in \Z$.  We note that $t.(xm) = (t.x)(t.m)$ for all $x \in \fg$ and $m \in K(\lambda)$.   Now by the definition of $\X_{K(\lambda)}$ and \cite[Proposition 5.2.1]{BKN1} it follows that when $K(\lambda)$ is considered as a $\langle y \rangle$-module, a trivial module appears as a direct summand; say it is spanned by $m \in K(\lambda)$.  We then check that $t.m$ spans a trivial direct summand of $K(\lambda)$ as a $\langle t.y \rangle$-module.  Hence $t.y \in \X_{K(\lambda)}$. Thus $\X_{K(\lambda)}$ is stable under the action of $\Z$ on $\fg_{\1}$. 
Since $\X_{K(\lambda)}$ is stable under the action of $t$ it follows that $t^{l}y = a^{-l}y_{-1} + a^{l}y_{1} \in \X_{K(\lambda)}$ for all $l > 0$.  Since $\X_{K(\lambda)}$ is also conical we can scale by $a^{l}$ and see that $y_{-1}+a^{2l}y_{1}\in \X_{K(\lambda)}$
for all $l > 1$.  However, as $\X_{K(\lambda)}$ is closed it follows by letting $l$ go to infinity that $y_{-1} \in \X_{K(\lambda)}$.  That is, by definition we have
\[
0 \neq y_{-1} \in  \V _{\fg_{-1}}(K(\lambda)).
\]  But this contradicts the fact that $\V _{\fg_{-1}}(K(\lambda)) = \{0 \}$ by \cite[Theorem 3.3.1]{BKN3}.  

For part (b), observe that we have proved that $c_{\mathcal F}\left(K^{\pm}(\lambda) \right)=\dim \X_{K^{\pm}(\lambda)}$ by part (a) and \cref{T:arbKacmodule}. The statement now follows 
by using the fact that $\V_{(\fg , \fg_{0})}\left(K^{\pm}(\lambda)\right )=\{ 0 \}$ by \cite[Corollary 3.3.1]{BKN2}. 
\end{proof}

\section{An Alternative to Support Varieties} \label{S:polytope}

\subsection{A Polytope Calculation}\label{SS:volumecalc} When computing the complexity of Kac modules a key ingredient is the lower bound provided by the dimension of support varieties.  However, for simple modules the known geometric tools are insufficient (see \cref{SS:gl11example}).  Thus we need to find a suitable replacement.  To do so, we use Ehrhart's theorem on counting lattice points in polytopes to obtain ``enough'' highest weights to provide an effective lower bound.

\begin{lemma}\label{L:CountingWeights} Fix an integer $k>1$. For any integer $d>0$, let $\tilde{S}(d)$ denote the set of all points $(b,a)=(b_{1}, \dotsc , b_{k}, a_{1}, \dotsc , a_{k}) \in \Z^{2k}$ which satisfy the following conditions.

First, we require the equality
\begin{equation}\label{E:ConditionP2}
(b_{1}+\dotsb + b_{k})-2(a_{1}+\dotsb + a_{k})  =  d.
\end{equation}  

In addition, we require the following inequalities to be satisfied:
\begin{equation} \label{E:Sinequalities}
\begin{aligned}
b_{u}-b_{u+1} &\geq d/2k^{2},  \\
b_{1} &\leq -d/2k^{2}, \\
a_{u}-a_{u+1} &\geq 0, \\
a_{1} &\leq 0,\\
0 \leq  ( b_{1}+\dotsb +b_{k})&-(a_{1}+\dotsb + a_{k})   \leq d, \\
a_{v} &\leq b_{v}
\end{aligned}
\end{equation}
 where $u=1, \dotsc , k-1$, and $v=1, \dotsc , k$.

 As a function of $d$, the number of elements of $\tilde{S}(d)$ is bounded below by a polynomial, $Q(d)$, of degree $2k-1$ with positive leading coefficient.
\end{lemma}

\begin{proof}  Let $\mathcal{H}$ be the affine hyperplane in $\R^{2k}$ defined by \cref{E:ConditionP2} when $d=1$.  Let $\mathcal{P}$ be the region in $\mathcal{H}$ defined by points in $\mathcal{H}$ which are simultaneous solutions to the inequalities \cref{E:Sinequalities} when $d=1$.  Then $\mathcal{P}$ is a polytope within $\mathcal{H}$ and hence is of dimension not more than $2k-1$.  To see that $\mathcal{P}$ is of dimension $2k-1$, it suffices to provide a point which simultaneously satisfies \cref{E:ConditionP2} when $d=1$ and \emph{strictly} satisfies the  inequalities \cref{E:Sinequalities} when $d=1$.  That is, that there exists points in the interior of $\mathcal{P}$.  To see that such a point exists, let $\delta > 1/2$ and $ 0< \delta'  <k$ be real numbers.  Let 
\begin{align*}
b&= \left( -(1+\delta)/k^{2}, -(1+2\delta)/k^{2}, \dotsc , -(1+k\delta)/k^{2}\right),\\
a&= \left( -(1+\delta+\delta')/k^{2}, -(1+2\delta+\delta')/k^{2}, \dotsc , -(1+k\delta+\delta')/k^{2}\right).
\end{align*} 
Then  for $u=1, \dotsc , k-1$ and  $v=1, \dotsc , k$,
\begin{gather*} 
b_{u}-b_{u+1} =  \delta/k^{2} > 1/2k^{2}, \\
b_{1} = -1/k^{2} - \delta/k^{2} < -1/2k^{2}, \\
a_{u}-a_{u+1} =    \delta/k^{2} > 0, \\
a_{1} = -(1+\delta+\delta')/k^{2} < 0, \\
0 < (b_{1}+\dotsb +b_{k})-(a_{1}+\dotsb +a_{k})= \delta'/k < 1,\\
a_{v} =  b_{v} - \delta'/k^{2} <  b_{v}.
\end{gather*} 
  Thus all the inequalities are strictly satisfied for any such $\delta, \delta'$.  Furthermore, 
\begin{align*}
(b_{1}+\dotsb +b_{k})-2(a_{1}+\dotsb a_{k}) &= [(1+\delta+2\delta')+ (1+2\delta+2\delta')+ \dotsc + (1+k\delta+2\delta')]/k^{2} \\
%&= \frac{\left(k +\frac{k(k+1)}{2}\delta +2k\delta' \right)}{k^{2}} \\
& = \frac{1}{k}\left(1+\frac{k+1}{2}\delta + 2\delta' \right).
\end{align*}  Since $k \ge 2$, one can choose (for example) $\delta=3/5,\ \delta'=(7k-13)/20$ so that this expression equals $1$.  Therefore there is an interior point in the polytope $\mathcal{P}$.

Now for any integer $d>0$, let $d\mathcal{P}$ be the dilated polytope
\[
d\mathcal{P}= \left\{(db_{1}, \dotsc , db_{k}, da_{1}, \dotsc , da_{k}) \mid (b_{1}, \dotsc , b_{k}, a_{1}, \dotsc , a_{k}) \in \mathcal{P} \right\}.
\]  We observe that the integer lattice points of $d\mathcal{P}$ are precisely those which satisfy the conditions of the lemma.  Let $L_{\mathcal{P}}(d)$ denote the number of such points within $d\mathcal{P}$; that is, the cardinality of $\tilde S(d)$.

As the coefficients of the hyperplanes defining $\mathcal{P}$ are rational, $\mathcal{P}$ is a rational polytope of dimension $2k-1$. By Ehrhart's theorem for rational polytopes (e.g., \cite[Theorem 3.23]{beckrobins}), $L_{\mathcal{P}}(d)$ is given by a quasipolynomial of degree $2k-1$.  That is, there is a fixed positive integer $M$ independent of $d$ and a sequence of polynomials $Q_{1}(d), \dotsc , Q_{M}(d)$ each of degree $2k-1$ such that when $d \equiv i \pmod M$,  $L_{\mathcal{P}}(d) = Q_{i}(d)$. Furthermore, the leading coefficient of each polynomial $Q_{1}(d), \dotsc , Q_{M}(d)$ is the volume of $\mathcal{P}$.  

From this we can construct a single polynomial of degree $2k-1$, $Q(d)$, such that $Q_{i}(d) \geq  Q(d)$ for all $d \in \Z_{>0}$; namely, for $j=0, \dotsc , 2k-1$ we can take the coefficient of $d^{j}$ in $Q(d)$ to be the minimum among the cofficients of $d^{j}$ among the polynomials $Q_{i}(d)$. In particular, note that the leading coefficient of $Q(d)$ will again be the volume of $\mathcal{P}$, and  hence, positive.

Therefore, there exists a degree $2k-1$ polynomial with positive leading coefficient such that 
\[
|\tilde S(d)| = L_{\mathcal{P}}(d) \geq Q(d)
\] for all $d \in \Z_{> 0}$. 
\end{proof}

\begin{remark}\label{R:polytopewhenkequals1}  We note that when $k=1$ the polytope $\mathcal{P}$ degenerates to a single point, $(-1,-1)$.   As a consequence, we treat $k=1$ as a separate case in the following arguments. 
\end{remark}

\subsection{The Su-Zhang Bijection}\label{SS:generalblock}  Su and Zhang define a bijection on highest weights between a block of $\F_{(\gl (m|n), \gl (m|n)_{0})}$ of atypicality $k$ and the principal block of $\F_{(\gl(k|k), \gl(k|k)_{0})}$ \cite{suzhang}.  We now use that bijection to define a set of pairs of highest weights which will provide an effective lower bound for the complexity of a simple module.  

\subsubsection{\texorpdfstring{The $k>1$ Case}{The k>1 Case}}\label{SSS:kgreater1} If $k>1$, then we can use the set $\tilde{S}(d)$ to define a set of pairs of highest weights in a certain block for $\gl (m|n)$ as follows.   Using \cref{P:serganovasblockdescription}, let us write $\B $ for the block of $\gl (m|n)$ of atypicality $k$ and with core 
\[
\left(\left\{2m-k,2m-k-2, \dotsc , k+2 \right\}, \left\{2m-k+2, 2m-k+4,\dotsc ,2m+2n-3k \right\} \right).
\]  If $k=m$ (resp.\ $k=n$), then we intend for the left hand (resp.\ right hand) set in the core to be empty.  In particular, in the case of $m=n=k$ we take $\B$ to be the principal block.

 Set $p:=m-k$ and $q:=2m-2k$.  Define an injective map
$\zeta : \R^{k} \to \bigoplus_{i=1}^{m+n} \R \varepsilon_{i}$ by 
\begin{multline*}
\zeta(x_{1}, \dotsc , x_{k})  = p\varepsilon_{1}+(p-1)\varepsilon_{2}+\dotsb + \varepsilon_{m-k} \\
                                   +  x_{1}\varepsilon_{m-k+1} +\dotsb +x_{k}\varepsilon_{m}  \\
   -x_{k}\varepsilon_{m+1}- \dotsb - x_{1}\varepsilon_{m+k} \\
    - (q+1)\varepsilon_{m+k+1} - (q+2)\varepsilon_{m+k+2} - \dotsb - (q+n-k)\varepsilon_{m+n}.
\end{multline*}
For example, we will frequently refer in what follows to the special weight
\begin{equation}\label{E:zetazero}
\nu := \zeta(0, \dotsc ,0) = p\varepsilon_{1} + \dotsb + \varepsilon_{m-k} - (q+1)\varepsilon_{m+k+1} - \dotsb - (q+n-k)\varepsilon_{m+n}.
\end{equation} Then $\nu \in X^{+}_{0}$ and $L(\nu)$ is a simple $\gl (m|n)$-module in $\B$.  More generally, observe that if $\zeta(x_{1}, \dotsc , x_{k}) \in X^{+}_{0}$, then $L(\zeta(x_{1}, \dotsc , x_{k}))$ lies in the block $\B$.

Define 
\[
\omega: \R^{2k} \to \bigoplus_{i=1}^{m+n} \R \varepsilon_{i} \times \bigoplus_{i=1}^{m+n} \R \varepsilon_{i}
\] by 
\[
\omega(x_{1}, \dotsc , x_{k}, y_{1}, \dotsc , y_{k}) = (\zeta(x_{1}, \dotsc , x_{k}), \zeta(y_{1}, \dotsc , y_{k})).
\] 
Then $\omega$ is clearly injective.
%Observe that in the case when $m=n=k$, then $\omega$ is simply given by 
%\[
%\omega(x_{1}, \dotsc , x_{k}, y_{1}, \dotsc , y_{k}) = \left( \sum_{i=1}^{k}x_{i}\varepsilon_{i}-\sum_{i=1}^{k}x_{k-i+1}\varepsilon_{k+i}, \sum_{i=1}^{k}y_{i}\varepsilon_{i}-\sum_{i=1}^{k}y_{k-i+1}\varepsilon_{k+i}\right)
%\]  

We define 
\[
S(d) = \omega(\tilde{S}(d)),
\] where $\tilde{S}(d)$ is the subset of $\Z^{2k}$ defined in \cref{L:CountingWeights}.  Then $S(d) \subset \B  \times \B$ and, since $\omega$ is injective, the cardinality of $S(d)$ equals the cardinality of $\tilde{S}(d)$.

We now introduce the bijection on highest weights defined in \cite[Theorem 3.29]{suzhang},
\begin{equation}\label{E:suzhangbijection}
\phi : \B  \to \B_{0,k|k},
\end{equation} where $\B_{0,k|k}$ denotes the principal block of $\gl (k|k)$.
 The interested reader will find the full definition in \cite{suzhang}.  However, we only require the value of this map on elements of $\B$ which appear in an element of $S(d)$. 
 For our purposes it suffices to note that for any $\zeta(x_{1}, \dotsc , x_{k}) \in \B$ we have
\[
\phi (\zeta (x_{1}, \dotsc , x_{k})) = x_{1}\varepsilon_{1} + \dotsb + x_{k}\varepsilon_{k} - x_{k}\varepsilon_{k} - \dotsb - x_{k}\varepsilon_{2k}.
\]  In particular, observe that for the weight $\nu$ defined in \cref{E:zetazero},
\[
\phi(\nu) = 0.
\]

It is easy to see using \cite[(3.13)]{suzhang} or by the definition of the Bruhat order in \cite{brundan1} that for all $\zeta(x_{1}, \dotsc , x_{k})$ and $\zeta(y_{1}, \dotsc , y_{k})$ which lie in $\B$ we have
\begin{equation}\label{E:suzhangpartialorder}
\zeta(x_{1}, \dotsc , x_{k}) \preccurlyeq \zeta(y_{1}, \dotsc , y_{k}) \text{ if and only if } \phi (\zeta(x_{1}, \dotsc , x_{k})) \preccurlyeq \phi (\zeta(y_{1}, \dotsc , y_{k})),
\end{equation}
and
\begin{equation}\label{E:suzhanglength}
l(\zeta(x_{1}, \dotsc , x_{k})) = l(\phi(\zeta(x_{1}, \dotsc , x_{k}))) = x_{1}+\dotsb + x_{k}.
\end{equation}

Now since $S(d) \subseteq \B \times \B$ we use \cref{E:Sinequalities}, \cref{E:zetazero}, \cref{E:suzhangpartialorder} and \cref{E:suzhanglength} to see the elements $(\mu, \sigma) \in S(d)$ satisfy

\begin{equation}\label{E:BrundanConditions}
\begin{gathered}
\sigma  \preccurlyeq \mu, \\
\sigma \preccurlyeq \nu, \\
0 \leq  l(\mu) - l(\sigma) \leq d;
\end{gathered}
\end{equation}
\medskip
\begin{equation}\label{E:SchmidtConditions} 
%a_{1}\geq a_{2}\geq \dotsb \geq a_{r}\geq 0 \\
-l(\sigma) = \frac{d-l(\mu)}{2};
\end{equation}
\medskip
\begin{equation}
\begin{aligned}
 - \mu_{m-k+1} & \geq d/2k^{2}, \\ 
\mu_{i}-\mu_{i+1} &\geq d /2k^{2}
\end{aligned}
\end{equation}
for $i=m-k+1, \dotsc , m-1$.

\subsubsection{\texorpdfstring{The $k=1$ Case}{The k=1 Case}}\label{SSS:kequal1}  We now consider the case when $k=1$.    Let $\B$ denote the block of $\gl(m|n)$ of atypicality one and with core 
\[
\left(\left\{2m-2, 2m-4, \dotsc , 2 \right\}, \left\{2m, 2m+2, \dotsc , 2m+2n-4 \right\} \right).
\]

Let $p$ and $q$ be as defined near the beginning of \cref{SSS:kgreater1} so that we have $p=m-1$ and $q=2m-2$.  For $d>6(m+n)$ and $2d/3 <  a \leq  d $, let  $b=a+m-n$ and set 
\[
\mu^{(a)}= a\varepsilon_{1} + p\varepsilon_{2} + \dotsb + \varepsilon_{m} - (q+1)\varepsilon_{m+1} - (q+2)\varepsilon_{m+2} - \dotsb - (q+n-1)\varepsilon_{m+n-1}-b\varepsilon_{m+n}.
\]

For $d>6(m+n)$ we then set
\[
S(d)= \left\{ (\mu^{(a)}, \mu^{(a)}) \mid a \in \Z, 2d/3 <  a \leq  d \right\}. 
\] Note that by our assumption on $d$ and the fact that $b \geq a >2d/3$, we have $S(d) \subset \B  \times \B$.

We will need the value of the Su-Zhang bijection $\phi : \B  \to \B_{0,1|1},$ on  weights of the form $\mu^{(a)}$. 
From the definition of $\phi$ it is easy to see that
\[
\phi \left( \mu^{(a)} \right) = (a-n+1)\varepsilon_{1} - (a-n+1)\varepsilon_{2}.
\] Finally, set $\nu \in \B$ to be the highest weight given by
\begin{equation}\label{E:nuzero}
\nu = \phi^{-1}(0).
\end{equation}   The interested reader who wishes to compute $\nu$ will need to refer to the definition of $\phi$ given in \cite{suzhang}.  However, for our purposes all we require is that it goes to $0$ under the bijection.

\subsubsection{}\label{SSS:suzhangonexts}  Finally, we record a crucial property of the Su-Zhang bijection.

\begin{lemma}\label{L:suzhangonexts}  Let $\phi : \B \to \B_{0,k|k}$ be the Su-Zhang bijection. Then for all $d \geq 0$ and all $\lambda, \mu\in \B$ we have
\begin{equation*}
\dim \Ext^{d}_{\F (\gl (m|n), \gl (m|n)_{0})}\left(K(\lambda), L(\mu) \right) = \dim \Ext^{d}_{\F (\gl (k|k), \gl (k|k)_{0})}\left(K(\phi (\lambda)), L(\phi (\mu)) \right).
\end{equation*}
\begin{equation*}
\dim \Ext^{d}_{\F (\gl (m|n), \gl (m|n)_{0})}\left(L(\lambda), L(\mu) \right) = \dim \Ext^{d}_{\F (\gl (k|k), \gl (k|k)_{0})}\left(L(\phi (\lambda)), L(\phi (\mu)) \right).
\end{equation*}

\end{lemma}

\begin{proof}  Both results follow from \cite[Theorem 3.29(2)]{suzhang} and \cite[Corollary 4.52]{brundan1}.
\end{proof}

\subsection{A Lower Bound on Dimensions of Projectives}\label{SS:lowerboundonprojgeneral}

\begin{lemma}\label{L:projlower}  Let $\B$ be the block given in the previous section and let $(\mu, \sigma) \in S(d) \subset \B  \times \B$.  Then, for $d$ sufficiently large,
\[
\dim P(\mu)  \geq  Cd^{(m+n-k-1)k },
\] where $C$ is a positive constant which is independent of $\mu$ and $\sigma$.
\end{lemma}
\begin{proof}   By \cref{E:boundingprojectives} it suffices to use the Weyl dimension formula to obtain a lower bound on the dimension of $L_{0}(\mu)$:  
\begin{equation}\label{E:Weyldim}
\dim L_{0}(\mu) = \prod_{\alpha \in \Phi_{m}^{+}}\frac{(\mu + \rho_{m}, \alpha)}{(\rho_{m}, \alpha)}  \prod_{\alpha \in \Phi_{n}^{+}}\frac{(\mu + \rho_{n}, \alpha)}{(\rho_{n}, \alpha)}.
\end{equation}

We first consider the case $k=1$ (and so $p=m-1$, $q=2m-2$, and $d>6(m+n)$).  From \cref{E:Weyldim} and the definition of $\mu = \mu^{(a)}$ we have
\begin{align*}
\dim L_{0}(\mu) &\geq \prod_{\substack{\alpha= \varepsilon_{1}-\varepsilon_{t} \\ t=2, \dotsc , m}}\frac{(\mu + \rho_{m}, \alpha)}{(\rho_{m}, \alpha)}  \prod_{\substack{\alpha= \varepsilon_{t}-\varepsilon_{m+n} \\ t=m+1, \dotsc , m+n-1}}\frac{(\mu + \rho_{n}, \alpha)}{(\rho_{n}, \alpha)} \\
&= \prod_{\substack{\alpha= \varepsilon_{1}-\varepsilon_{t} \\ t=2, \dotsc , m}}\frac{a-(p-t+2)+(\rho_{m}, \alpha)}{(\rho_{m}, \alpha)}  \prod_{\substack{\alpha= \varepsilon_{t}-\varepsilon_{m+n} \\ t=m+1, \dotsc , m+n-1}}\frac{b- (q+t-m) +(\rho_{n}, \alpha)}{(\rho_{n}, \alpha)}.
\end{align*}  However, since $a > 2d/3$ and $d/3 \geq p$ we have 
\[
a- (p-t+2) \geq a- p > 2d/3 - d/3 = d/3 
\]
Similarly, since $b > 2d/3$ and $d/3 \geq q+n-1$ we have
\[
b - (q+t-m) \geq b-(q+n-1) > 2d/3 - d/3 = d/3.
\]
Substituting yields
\begin{align*}
\dim L_{0}(\mu)  &\geq \prod_{\substack{\alpha= \varepsilon_{1}-\varepsilon_{t} \\ t=2, \dotsc , m}}\frac{d/3+(\rho_{m}, \alpha)}{(\rho_{m}, \alpha)}  \prod_{\substack{\alpha= \varepsilon_{t}-\varepsilon_{m+n} \\ t=m+1, \dotsc , m+n-1}}\frac{d/3+( \rho_{n}, \alpha)}{(\rho_{n}, \alpha)} \\
&\geq C d^{m-1}d^{n-1}=Cd^{m+n-2},
\end{align*} where $C$ is a constant independent of $\mu$ and $\sigma$.  This proves the desired result when $k=1$.

We now consider the case $k>1$.  We first study the first factor in \cref{E:Weyldim}.  We have 
\begin{equation}\label{E:tripleproduct2}
 \prod_{\alpha \in \Phi_{m}^{+}}\frac{(\mu + \rho_{m}, \alpha)}{(\rho_{m}, \alpha)}=  \prod_{\alpha \in A_{m}}\frac{(\mu + \rho_{m}, \alpha)}{(\rho_{m}, \alpha)} \prod_{\alpha \in B_{m}}\frac{(\mu + \rho_{m}, \alpha)}{(\rho_{m}, \alpha)} \prod_{\alpha \in C_{m}}\frac{(\mu + \rho_{m}, \alpha)}{(\rho_{m}, \alpha)},
\end{equation} where $A_{m}, B_{m}, C_{m}$ are defined in \cref{SS:bruhat}.

If we let $Z = \max\{ (\rho_{m}, \alpha) \mid \alpha \in \Phi_{m}^{+} \}$, then we have 
\[
\prod_{\alpha \in A_{m}}\frac{(\mu + \rho_{m}, \alpha)}{(\rho_{m}, \alpha)} \geq \prod_{\alpha \in A_{m}}\frac{(\mu + \rho_{m}, \alpha)}{Z}=:C_{0}.
\] But by the definition of $A_{m}$ the value of $C_{0}$ depends only on the core of $\mu$ and not on $\mu$ itself.  Hence it depends only on $m, n$ and $\B $.

For $(\mu, \sigma) \in S(d)$ we have
\[
(\mu + \rho_{m}, \alpha) \geq (\mu, \alpha) \geq d/2k^{2}
\] for all $\alpha \in B_{m}$ and $\alpha \in C_{m}$. Taken together with the fact that $B_{m}$ has cardinality $(m-k)k$ and $C_{m}$ has cardinality $k(k-1)/2$, we see that
\[
\prod_{\alpha \in \Phi_{m}^{+}}\frac{(\mu + \rho_{m}, \alpha)}{(\rho_{m}, \alpha)} \geq C_{1}d^{(m-k)k+k(k-1)/2}
\] for some constant $C_{1}$ which is independent of $\mu$ and $\sigma$.  Similarly,
\[
\prod_{\alpha \in \Phi_{n}^{+}}\frac{(\mu + \rho_{m}, \alpha)}{(\rho_{m}, \alpha)} \geq C_{2}d^{(n-k)k+k(k-1)/2}.
\]

Combining these we see that 
\[
\dim P(\mu) \geq C_{0}C_{1}C_{2}d^{(m-k)k + k(k-1)/2 +(n-k)k +k(k-1)/2}= Cd^{(m+n-k-1)k},
\] for some constant $C$ which is independent of $\mu$ and $\sigma$, as desired.
\end{proof}

\section{\texorpdfstring{Complexity for Simple Modules}{Complexity for Simple Modules}} \label{S:simplecomplexity}

\subsection{} We first observe that for a fixed $\gl(m|n)$  any two simple modules with the same atypicality have the same complexity.

\begin{theorem}\label{T:constantatypicality}  Let $L(\lambda)$ and $L(\mu)$ be two simple modules for $\gl (m|n)$ with $\atyp (\lambda) = \atyp (\mu)$. Then 
\begin{equation}\label{E:splitting}
L(\lambda)^{*}\otimes L(\lambda) \otimes L(\mu) \cong L(\mu) \oplus U
\end{equation}
for some $\gl (m|n)$-module $U$.

Furthermore, the complexity of $L(\lambda)$ equals the complexity of $L(\mu)$.

\end{theorem}

\begin{proof}  Let $L(\lambda)$ and $L(\mu)$ be two simple modules with the same atypicality.  Let $P_{\bullet} \to L(\lambda)$ be a minimal projective resolution of $L(\lambda)$.  Tensoring this resolution on the left by $L(\lambda)^{*}$ and on the right by $L(\mu)$ we obtain a (not necessarily minimal) projective resolution of $L(\lambda)^{*} \otimes L(\lambda)\otimes L(\mu)$ with rate of growth equal to the rate of growth of $P_{\bullet}$.  Therefore, we deduce that
\begin{equation}\label{E:GKW1}
c_{\F }\left(L(\lambda) \right) \geq c_{\F }\left( L(\lambda)^{*} \otimes L(\lambda) \otimes L(\mu)\right).
\end{equation}

By \cite[Corollary~6.6]{serganova4} every simple $\gl (m|n)$-module admits an ambidextrous trace in the sense of \cite{GKP}.  In particular, $L(\mu)$ has an ambidextrous trace and by \cite[Theorem~3.3.2]{GKP} this trace defines a modified dimension function, $\md_{L(\mu)}$, on the ideal of $\F$ generated by $L(\mu)$.  By the generalized Kac-Wakimoto conjecture, stated for basic classical Lie superalgebras in \cite[Conjecture~6.3.2]{GKP} and proven for $\gl (m|n)$ in \cite[Corollary~6.7]{serganova4}, it follows that $L(\lambda)$ is in the ideal generated by $L(\mu)$ and $\md_{L(\mu)}\left( L(\lambda)\right) \neq 0$.  However, by \cite[Corollary~4.3.3]{GKP} this implies  the canonical surjection induced by the evaluation map,
\begin{equation*}
L(\lambda)^{*}\otimes L(\lambda) \otimes L(\mu) \to L(\mu),
\end{equation*}
splits.  In short, because $L(\lambda)$ and $L(\mu)$ have the same atypicality we have 
\[
L(\lambda)^{*}\otimes L(\lambda) \otimes L(\mu) \cong L(\mu) \oplus U
\] for some $\gl (m|n)$-module $U$ (where the isomorphism preserves the $\Z_{2}$-grading).  Using \cref{P:complexityext} and the additivity of $\Ext$, we then see that 
\begin{equation}\label{E:GKW2}
c_{\F }\left( L(\lambda)^{*}\otimes L(\lambda) \otimes L(\mu) \right) \geq c_{\F } \left(L(\mu) \right).
\end{equation}  Combining \cref{E:GKW1} and \cref{E:GKW2}, we obtain 
\[
c_{\F }\left(L(\lambda) \right) \geq c_{ \F}\left(L(\mu) \right).
\]  The argument is symmetric under switching $L(\lambda)$ and $L(\mu)$ and so we have equality.
\end{proof}

\subsection{} We now compute the complexity of a simple $\gl (m|n)$-module of atypicality $k$. The end result will be the following theorem.  

\begin{theorem}\label{T:simplemodulecomplexity} Let $L(\lambda)$ be a simple $\gl (m|n)$-module of atypicality $k$.  Then \[
c_{\mathcal F}(L(\lambda))=\dim  \overline{(\fg_{1})_{k}} +k =  (m+n)k-k^{2}+k = \dim \X_{L(\lambda)} + \dim \V_{(\fg ,\fg_{0})}(L(\lambda)).
\]
\end{theorem}

By \cref{T:constantatypicality} it suffices to compute complexity for the simple module $L(\nu)$ in the block $\B$ given in \cref{SS:generalblock}; recall the definition of $\nu$ in \cref{E:zetazero} for $k>1$ and \cref{E:nuzero} for $k=1$. The second equality of \cref{T:simplemodulecomplexity} follows from the well known formula for the dimension of the variety $\overline{(\fg_{1})_{k}}$  \cite[Proposition 12.2]{Har:92}.  The third equality follows from \cite[Theorems 4.5 and 5.4]{dufloserganova} and \cite[Theorem 4.8.1]{BKN2}.  Thus we focus on computing the complexity of $L(\nu)$ by computing sharp upper and lower bounds for the expression given by \cref{P:complexityext}.  In particular, we see that combining \cref{L:simpleupperbound} and \cref{L:simplelowerbound} (below) proves the first equality of \cref{T:simplemodulecomplexity}.  

\subsection{The Upper Bound}\label{SS:generalupperbound}

We first prove an intermediate result.

\begin{lemma}\label{L:simpleextbound} Let $L(\nu)$ and $\B$ be as above. Then 
\begin{equation*}
 \dim \Ext_{(\fg , \fg_{0})}^{d}(L(\nu),\bigoplus_{\mu\in \B } L(\mu)) \leq  Dd^{2k-1},
\end{equation*}
where $D$ is a positive constant. 
\end{lemma}

\begin{proof}

By our choice of $\nu$ we may apply \cref{L:suzhangonexts} and assume without loss that 
$\gl (m|n) = \gl (k|k)$, $L(\nu) =  \C$, and $\B=\B_{0}$ is the principal block of $\F$.  We first analyze 
an individual term in the direct sum.  Fix $\mu\in \B_{0}$.  We have
\begin{align}\label{E:productreduction}
\dim \Ext^{d}_{(\g,\g_{0})}(\C,L(\mu))  &=  \sum_{i+j=d}\ \sum_{\s\in \B_{0}} \dim \Ext^{i}_{\F}(K(\s),\C)\,\dim \Ext^{j}_{\F}(K(\s),L(\mu))\notag\\
  &=  \sum_{i+j=d}\ \sum_{\s\in \B_{0}} \dim \Ext^{i}_{({\mathfrak p}^{+},{\mathfrak g}_{0})}(L_{0}(\s),\C)\,\dim \Ext^{j}_{\F}(K(\s),L(\mu)) \\ 
  &=  \sum_{i+j=d}\ \sum_{\s\in \B_{0}} \dim \Hom_{\g_{0}}(L_{0}(\s),S^{i}(\fg_{1}^{*}))\,\dim \Ext^{j}_{\F}(K(\s),L(\mu)), \notag
\end{align}
where the first line is by \cite[Theorem 4.51 and Corollary 4.52]{brundan1}, the second line is by Frobenius reciprocity, and in the third, 
by using a spectral sequence argument (cf.\ \cite[(3.3.2)]{BKN3}). 

Suppose there is a nonzero term in the last sum. Then by \cite{Sch:69} the $\Hom$-space is one dimensional and $\s\preceq 0$ with $i=-l(\s)$ (recalling that in the principal block of $\gl (k|k)$ we have $l(\sigma)=|\sigma|$ and the alternate description of the Bruhat order). Also $\s\preceq\mu$ and $j=l(\mu)-l(\s)-b$ where $b\in \{\, 0, 1,\dots, \dim \fg_{-1}=k^{2} \,\}$ (cf.\ \cref{T:finitenumberofKLpolys} and \cite[Theorem 4.5.1]{brundan1}). In particular, 
\begin{equation} \Label{E:LambdaLowerBound}
l(\mu) \ge l(\s) = -i \ge -d.
\end{equation}
Also, 
\[
l(\mu)+i-k^{2} = l(\mu)-l(\s)-k^{2} \le j = d-i  \le l(\mu)-l(\s) = l(\mu)+i; 
\] which in turn implies
\begin{equation} \Label{E:iBounds}
 \frac{d-l(\mu)}{2} \le i \le \frac{d+k^{2}-l(\mu)}{2}.
\end{equation}
But the last inequality together with $i\ge 0$ implies $l(\mu) \le d+k^{2}$. Combining this with \cref{E:LambdaLowerBound}, we have
\begin{equation} \Label{E:LambdaBounds}
-d \le l(\mu) \le d+k^{2}.
\end{equation}

We will also require another estimate on the entries of $\mu$. From the conditions $\s \preceq 0$, $\s \preceq \mu$, $i=l(0)-l(\s)$, $j\ge l(\mu)-l(\s)-k^{2}$, and $i+j=d$, we deduce that
\begin{equation} \Label{E:SigmaLengthEstimate}
[l(0)-l(\s)] + [l(\mu)-l(\s)] \le d+k^{2}.
\end{equation}
There is a greatest element $\s^{0}$ of $X^{+}_{0}$ dominated by both $0$ and $\mu$ in the Bruhat order, with coordinates defined by $\s^{0}_{j}=\min(0,\mu_{j}),\ 1\le j\le k$. Note that $[l(0)-l(\s^{0})] + [l(\mu)-l(\s^{0})] = \sum_{j=1}^{k} |\mu_{j}|$. Since $\s\preceq\s^{0}$, \cref{E:SigmaLengthEstimate} implies that
\begin{equation} \Label{E:AbsLambdaBounds}
\sum_{j=1}^{k} |\mu_{j}| \le d+k^{2}.
\end{equation}

So we may assume $\mu$ satisfies \cref{E:LambdaBounds} and \cref{E:AbsLambdaBounds}. Taking into account \cref{E:iBounds}, and the fact that $\dim \Ext^{j}_{\F}(K(\s),L(\mu))$ is the coefficient of a Kazhdan-Lusztig polynomial, which by \cref{T:finitenumberofKLpolys}  is bounded by a constant $C_{0}$, we have
\[
\dim \Ext^{d}_{(\g,\g_{0})}(\C,L(\mu)) \le C_{0} \sum_{i} \#\{\, 0\succeq\s\in X^{+}_{0} \mid l(\s)=-i\,\},
\]
where the sum is over $\max(0,\frac{d-l(\mu)}{2}) \le i \le\min(d, \frac{d+k^{2}-l(\mu)}{2})$. Now $0 \succeq \s \in X^{+}_{0}$ and $l(\s)=-i$ means that $-\s$ is a partition of $i$ into at most $k$ (positive) parts. And there is a constant $C_{1}$ (depending only on $k$) such that the number of such partitions is at most $C_{1} i^{k-1}$  \cite[Corollary 15.1]{nathanson}. Thus
\begin{equation}\label{E:bound}
\dim \Ext^{d}_{(\g,\g_{0})}(\C,L(\mu)) \le C_{0}  C_{1 }\sum_{i} i^{k-1},
\end{equation}
with the same conditions on $i$ as before. Assuming without loss of generality that $2d \ge k^{2}$, the last expression is maximized when $l(\mu)=-d+k^{2}$, giving
\[
\dim \Ext^{d}_{(\g,\g_{0})}(\C,L(\mu)) \le C_{0}  C_{1 }\sum_{d-k^{2}/2\le i\le d} i^{k-1} \le C_{2}d^{k-1}
\]
for some constant $C_{2}$ depending only on $k$.  Alternatively, to bound \cref{E:bound} it suffices to note that for $d$ sufficently large the number of terms in the sum is bounded by a constant independent of both $d$ and $\mu$, and that $i \leq (d+k^{2}-l(\mu))/2 \leq d + k^{2}/2$  (using \cref{E:LambdaBounds} to obtain the second inequality).

Lastly, in order for \cref{E:productreduction} to be nonzero the condition \cref{E:AbsLambdaBounds} certainly implies that each $|\mu_{j}| \le d+k^{2}$ for $1\le j\le k$, so the total number of such $\mu$ in $\B_{0}$ is bounded by $C_{3} d^{k}$ for some constant $C_{3}$ depending only on $k$.

Putting the ingredients together, we have
\begin{align*}
\dim \Ext^{d}_{(\g,\g_{0})}(\C,\bigoplus_{\mu \in \B_{0}}  L(\mu))
&= \sum_{\mu \in \B_{0} }  \dim \Ext^{d}_{\F}(\C,L(\mu))\\
&\le \sum_{\substack{\mu \in \B_{0}, \\ \sum |\mu_{k}|\le d+k^{2}}} C_{2}d^{k-1}\\
&\le C_{3}C_{2}d^{k}d^{k-1}\\
&=D d^{2k-1}. \qedhere
\end{align*}
\end{proof}

We now prove an upper bound for the complexity of $L(\nu)$.
\begin{prop}\label{L:simpleupperbound} Let $L(\nu)$ and $\B$ be as above. Then for all $d$ we have
\begin{equation*}
 \dim \Ext_{(\fg , \fg_{0})}^{d}(L(\nu),\bigoplus_{\mu\in \B } L(\mu)^{\dim P(\mu)}) \leq Kd^{{(m+n-k+1)k-1}},
\end{equation*}
where $K$ is a positive constant.
\end{prop}

\begin{proof}  Recall that if $P_{\bullet} \to L(\nu)$ is a minimal projective resolution, then since 
\[
\Ext^{d}_{\F}(L(\nu), L(\mu)) \cong \Hom_{\F} (P_{d}, L(\mu))
\] we have that this vector space is nonzero if and only if $P(\mu)$ is a direct summand of $P_{d}$.  
Therefore, combining \cref{L:simpleextbound} with \cref{T:Projupperbound}, we obtain 
\begin{align*}
\dim \Ext^{d}_{(\fg,\fg_{0}) }(L(\nu), \bigoplus_{\mu \in \B} L(\mu)^{P(\mu)} ) 
%                            &= \sum_{\mu \in \B } \dim P(\mu) \dim \Ext^{d}_{\F }(L(\nu), L(\mu))\\
%                            &\leq \sum_{\mu \in \B } C d^{(m+n-1)k-k^{2}} \dim  \Ext^{d}_{\F }(L(\nu), L(\mu)) \\
                            & \leq C d^{(m+n-k-1)k} \dim \Ext^{d}_{(\fg , \fg_{0}) }(L(\nu), \bigoplus_{\mu\in \B } L(\mu)) \\
                            &\leq  CD d^{(m+n-k-1)k}d^{2k-1} \\
                            &= K d^{(m+n-k+1)k-1}. \qedhere
\end{align*}
\end{proof}

\subsection{The Lower Bound}\label{SS:simplelowerbound}

We now compute a lower bound for the complexity of $L(\nu)$. 

\begin{prop}\label{L:simplelowerbound} Let $L(\nu)$ and $\B$ be as above. Then as a function of $d$ for all $d$ sufficiently large
\begin{equation*}
 \dim \Ext_{(\fg , \fg_{0})}^{d}(L(\nu),\bigoplus_{\mu\in \B } L(\mu)^{\dim P(\mu)})
\end{equation*}
is bounded below by a polynomial of degree ${(m+n-k+1)k-1}$ with positive leading coefficient.
\end{prop}

\begin{proof} 
We first consider the case when $k=1$.  Let $P_{\bullet} \to \C$ be the minimal projective resolution of $\C$ as a $\gl(1|1)$-module given in \cite{BKN3}.  Combining the basic properties of a minimal projective resolution with \cref{L:suzhangonexts} and \cref{L:projlower} we have
\begin{align*}
\dim &\Ext^{d}_{(\g,\g_{0})}(L(\nu),\bigoplus_{\mu \in \B} L(\mu)^{\dim P(\mu)}) \notag\\
  &\geq  \sum_{(\mu, \mu) \in S(d)} \dim P(\mu) \dim \Ext^{d}_{\F }(L(\nu), L(\mu)) \notag\\ 
    &=  \sum_{(\mu, \mu) \in S(d)} \dim P(\mu) \dim \Ext^{d}_{\F_{(\gl(k|k), \gl(k|k)_{0})} }(L(\phi (\nu)), L(\phi(\mu))) \notag\\ 
    &\geq   Cd^{(m+n-k)k -k} \sum_{(\mu, \mu) \in S(d)} \dim \Ext^{d}_{\F_{(\gl(k|k), \gl(k|k)_{0})} }(\C, L(\phi(\mu))) \notag\\ 
     &=  Cd^{(m+n-k)k -k}\sum_{(\mu, \mu) \in S(d)} \dim \Hom_{\F_{(\gl(k|k), \gl(k|k)_{0})}}(P_{d}, L(\phi(\mu))). \notag
\end{align*}
From the construction of $P_{d}$ and $S(d)$ we can use \cref{SSS:kequal1} to see that this $\Hom$ space is nonzero for each $\mu = \mu^{(a)}$ when $a-n+1$ has the same parity as $d$.  Therefore the total dimension of the $\Hom$ spaces given above is bounded below by a linear function in $|S(d)|$ which, in turn, is bounded below by a linear function in $d$.  This implies the desired result for $k=1$.

We now consider the case $k>1$. As in \cref{E:productreduction} we have
\begin{align}\label{E:productreduction2}
\dim &\Ext^{d}_{(\g,\g_{0})}(L(\nu),\bigoplus_{\mu \in \B } L(\mu)^{\dim P(\mu)}) \notag\\
%  &= \sum_{\mu \in \B } \dim P(\mu) \dim\Ext^{d}_{(\g,\g_{0})}(\C, L(\mu)) \notag  \\ 
%  &=  \sum_{\mu, \sigma \in \B} \sum_{i+j=d}   \dim P(\mu) \dim \Ext^{i}_{\F}(K(\s),\C)\,\dim \Ext^{j}_{\F}(K(\s),L(\mu)) \notag\\
% & \geq  \sum_{(\mu, \sigma) \in S(d)} \sum_{i+j=d}  \dim P(\mu) \dim \Ext^{i}_{\F}(K(\s),\C)\,\dim \Ext^{j}_{\F}(K(\s),L(\mu)) \notag \\
  & =\sum_{i+j=d}\ \sum_{\mu, \s\in \B} \dim P(\mu) \dim \Ext^{i}_{\F}(K(\s),L(\nu))\,\dim \Ext^{j}_{\F}(K(\s),L(\mu))\notag\\
   &\geq \sum_{i+j=d}\ \sum_{(\mu, \sigma) \in S(d)} \dim P(\mu)\dim \Ext^{i}_{\F}(K(\s),L(\nu))\,\dim \Ext^{j}_{\F}(K(\s),L(\mu)) 
%  & = \sum_{\mu, \sigma \in \B } \sum_{i+j=d} \dim P(\mu) \dim \Hom_{\g_{0}}(L_{0}(\s),S^{i}(\fg_{1}^{*}))\,\dim \Ext^{j}_{\F}(K(\s),L(\mu)) \notag\\
%  & \geq \sum_{(\mu, \sigma) \in S(d)} \sum_{i+j=d} \dim P(\mu) \dim \Hom_{\g_{0}}(L_{0}(\s),S^{i}(\fg_{1}^{*}))\,\dim \Ext^{j}_{\F}(K(\s),L(\mu)). 
\end{align}

 Each $(\mu, \sigma) \in S(d)$ satisfies conditions \cref{E:BrundanConditions} so by \cite[Theorem 4.51]{brundan1} we have that 
\[
\dim \Ext^{l(\mu)-l(\sigma)}_{\F}(K(\s),L(\mu)) = 1.
\]  By \cref{L:suzhangonexts} and the argument used in \cref{E:productreduction} we have
\begin{align}\label{E:SchmidApp}
\dim \Ext^{d-l(\mu)+l(\sigma)}_{\F}(K(\s),L(\nu)) &= \dim \Ext^{d-l(\mu)+l(\sigma)}_{\F_{(\gl (k|k), \gl(k|k)_{0})}}(K(\phi (\s)),\C ) \notag\\
   &=  \dim \Hom_{\gl (k|k)_{0}}(L_{0}(\phi (\sigma)),S^{d-l(\mu)+l(\sigma)}(\fg_{1}^{*})).
\end{align}
The element $(\mu, \sigma) \in  S(d)$ also satisfies \cref{E:SchmidtConditions} and so using \cref{E:suzhangpartialorder} and \cref{E:suzhanglength} to translate \cref{E:SchmidtConditions} to the analogous conditions on $(\phi (\mu), \phi (\sigma))$ we may apply  \cite{Sch:69} to \cref{E:SchmidApp} and obtain
\[
 \dim \Ext^{d-l(\mu)+l(\sigma)}_{\F}(K(\s),L(\nu))=1 
\]  
Taken together, we see that
\begin{equation*}
\sum_{i+j=d}  \sum_{(\mu, \sigma) \in S(d)}  \dim \Ext^{i}_{\F}(K(\s),L(\nu))\,\dim \Ext^{j}_{\F}(K(\s),L(\mu)) \geq |S(d)|.
\end{equation*}  Applying this along with \cref{L:projlower} to \cref{E:productreduction2} yields 
\begin{align*}
\dim \Ext^{d}_{(\g,\g_{0})}(L(\nu),\bigoplus_{\mu \in \B} L(\mu)^{\dim P(\mu)})  \geq Cd^{(m+n-k)k -k}|S(d)| \geq  Cd^{(m+n-k)k -k}Q(d).
\end{align*}
The last inequality follows from the fact that the cardinality of $S(d)$ equals the cardinality of $\tilde{S}(d)$ and so is bounded below by a polynomial, $Q(d)$, of degree $2k-1$ with positive leading coefficient, by \cref{L:CountingWeights}.  This proves the desired result when $k>1$.
\end{proof}

\section{A Categorical Invariant}\label{S:zcomplexity} 
\subsection{}\label{SS:zcomplexity} 
We will first assume that $\fg$ is a classical Lie superalgebra and $M$ is a module in $\F = \F_{(\fg ,\fg_{\0})}$.  It is natural to consider 
\[
z_{\F}(M) = r\left(\Ext^{\bullet}_{(\g,\g_{\0})}(M,\bigoplus  S)\right),
\] where the direct sum runs over all simple modules of $\F$.  Unlike complexity, $z_{\F}(-)$ has the advantage of being invariant under category equivalences.  

Using our complexity calculations, we compute this invariant for the Kac, dual Kac, and simple modules of $\gl(m|n)$.

\begin{theorem}\label{T:zcomplexity} Let $\fg =\gl (m|n)$ and let $X(\lambda)$ (resp.\  $L(\lambda)$) be a Kac or dual Kac module (resp.\ simple module) of atypicality $k$ in $\F$.  Then, 
\[
z_{\F}\left(X(\lambda) \right) = k
\] and 
\[
z_{\F}\left(L(\lambda) \right) = 2k.
\]
\end{theorem}

\begin{proof} We consider $K(\lambda)$; the proof for $K^{-}(\la)$ is the same. By  \cref{E:Kacboundreduction} there is a positive constant $K$ such that for all $d \geq 1$
\[
\dim \Ext^{d}_{(\g,\g_{0})}(K(\lambda),\bigoplus  S) \leq Kd^{k-1}.
\]  On the other hand, if the left hand side could be bounded above by $K'd^{k-2}$ for some positive constant $K'$, then this would imply by the proof of \cref{T:arbKacmodule} that the complexity of $K(\lambda)$ is strictly less than $(m+n)k-k^{2}$, contradicting the conclusion of that proof.  Therefore, the power $k-1$ is sharp and $z_{\F}\left(K(\lambda) \right)=k$.

An identical argument applies to $L(\lambda)$ using \cref{L:simpleextbound}. 
\end{proof}

\subsection{}\label{SS:fcomplexity} Recalling that we assume $n \leq m$, we set $\ff_{\1}\subset \fg_{\1}$ to be the span of the matrix units $E_{m-t+1, m+t}$ and $E_{m+t, m-t+1}$ for $t=1, \dotsc , n$.  Set $\ff_{0} = \ff_{\0} = [\ff_{\1}, \ff_{\1}]$.  We then define a subalgebra of $\fg$ by 
\[
\ff : = \ff_{\0}\oplus \ff_{\1}.
\] The Lie superalgebra $\ff$ is classical (and Type I) and so has a support variety theory.  Furthermore, as $[\ff_{0}, \ff_{\1}] = 0$ it follows that these varieties admit a rank variety description and, in particular, can be identified as  subvarieties of $\ff_{\1}$. 

The subalgebra $\ff$ is a ``detecting'' Lie subsuperalgebra of $\fg$ which first appeared in \cite{BKN1} and can be seen to have a remarkable cohomological detection property due to work of Lehrer, Nakano, and Zhang \cite{LNZ} (where it is called $\tilde{\ff}$).   We now show that these detecting subalgebras naturally capture the above categorical invariant for Kac, dual Kac, and simple modules.

\begin{theorem}\label{T:fdetectszcomplexity}  Let $\fg = \gl (m|n)$ and let $X(\lambda)$ denote a Kac, dual Kac, or simple module in $\F=\F_{(\fg, \fg_{0})}$. Then 
\[
z_{\F}\left(X(\lambda) \right) = \dim \V_{(\ff , \ff_{0})}\left(X(\lambda) \right) = c_{\F_{(\ff , \ff_{0})}}\left(X(\lambda) \right).
\]

\end{theorem}

\begin{proof}  The second equality is immediate from \cite[Theorem 2.9.1(c)]{BKN3} and \cite[Proposition 5.2.2]{BKN2}.  We now consider the first equality.

We first obtain an upper bound on $\dim \V_{(\ff, \ff_{0})}(X(\lambda))$.  Let $y \in \V_{(\ff, \ff_{0})}\left( X(\lambda)\right)$ and choose $y_{-1} \in \ff_{\1} \cap \fg_{-1}$ and $y_{1} \in \ff_{\1} \cap \fg_{1}$ so that $y=y_{-1}+y_{1}$.  Since $X(\lambda)$ inherits a $\Z$-grading from the $\Z$-grading on $\fg$, the argument using the $\Z$-action on $X(\lambda)$ in the proof of \cref{T:Kacvarieties} shows that $y_{-1}, y_{1} \in  \V_{(\ff, \ff_{0})}\left(X(\lambda) \right)$.  From this we conclude that 
\begin{equation}\label{E:Bound}
\V_{(\ff, \ff_{0})}\left( X(\lambda) \right)\subseteq \left( \ff_{\1} \cap \V_{\fg_{-1}}(X(\lambda) \right) \times \left( \ff_{\1} \cap \V_{\fg_{1}}(X(\lambda) \right).
\end{equation}

If $X(\lambda)$ is a Kac module we can use \cref{T:Kacvarieties} and \cite[Theorem 3.3.1]{BKN3} to deduce $\dim \left(  \ff_{\1} \cap \V_{\fg_{1}}(X(\lambda))\right)=k$ and  $\dim  \left( \ff_{\1} \cap \V_{\fg_{-1}}(X(\lambda))\right)=0$ and, hence, the dimension of the right hand variety in \cref{E:Bound}  is $k$.  The dual Kac module is handled similarily using  \cref{T:Kacvarieties} and \cite[Theorem 3.3.2]{BKN3}.  

If $X(\lambda)$ is a simple module, then we use \cite[(3.8.1)]{BKN3} (which ultimately depends on calculations in \cite{dufloserganova}) to deduce that $\dim  \left( \ff_{\1} \cap \V_{\fg_{i}}(X(\lambda))\right)=k$ for $i=-1,1$ and so the dimension of the right hand variety in \cref{E:Bound} is $2k$.

We next obtain a lower bound. From the rank variety description it is clear that 
\begin{equation}\label{E:Bound2}
\left( \ff_{\1} \cap \V_{\fg_{-1}}(X(\lambda) \right) \cup \left( \ff_{\1} \cap \V_{\fg_{1}}(X(\lambda) \right) \subseteq \V_{(\ff, \ff_{0})}(X(\lambda)). 
\end{equation}  If $X(\lambda)$ is a Kac or dual Kac module then by the above calculations the left hand variety is $k$-dimensional. Therefore, the dimension of $ \V_{(\ff, \ff_{0})}\left( X(\lambda)\right)$ equals $k$ and applying \cref{T:zcomplexity} proves the theorem for Kac and dual Kac modules.

Now if $X(\lambda)=L(\lambda)$ is a simple module and $L(\gamma)$ is another simple $\gl (m|n)$-module of atypicality $k$, then by \cref{E:splitting} and the basic properties of rank varieties \cite[Proposition 6.3.1 and Theorem 6.4.2]{BKN1} we have 
\[
\V_{(\ff , \ff_{0})}(L(\lambda)) = \V_{(\ff, \ff_{0})}(L(\gamma)).
\]     In particular, applying this to the block equivalences of Serganova introduced in \cref{SS:serganovablockequivalence} we see that we may assume without loss that $\Res_{\mu} L(\lambda)$ is isomorphic to the trivial module for the subalgebra $\fg' \cong \gl (k|k)$.  If we set $\ff ' = \ff \cap  \fg '$, then $\ff '$ is the corresponding detecting subalgebra for $\fg'$.  By the rank variety description it is immediate that 
\[
\ff'_{\1}= \V_{(\ff ', \ff '_{0})}(\Res_{\mu} L(\lambda)) \subseteq \V_{(\ff ', \ff '_{0})}(L(\lambda))\subseteq \V_{(\ff , \ff _{0})}(L(\lambda)).
\]  Thus, the dimension of $ \V_{(\ff , \ff _{0})}(L(\lambda))$ is at least $2k$ and, hence, equals $2k$.  Combining this with \cref{T:zcomplexity} proves the theorem for simple modules.
\end{proof}

\subsection{\texorpdfstring{The $\gl (1|1)$ Case}{The gl(1|1) Case}}\label{SS:complexity}  We now show that the $\ff$ support varieties capture the invariant defined in the previous section and the complexity 
for an arbitrary $\gl (1|1)$-module.  

\begin{theorem}\label{T:gl11zcomplexity} Let ${\mathfrak g}=\gl (1|1)$ and $M\in {\mathcal F}_{(\g,\g_{0})}$. Then 
$$z_{\F}\left(M\right)= c_{\mathcal F}(M)=\dim \V_{(\ff, \ff_{0})}(M).$$
\end{theorem} 

\begin{proof} We may reduce to the case when $M$ is in the principal block ${\mathcal B}_{0}$ of 
${\mathcal F}$ otherwise $M$ will be projective. The simple modules of the principal block, $L(\lambda):=(\lambda\epsilon_{1}-\lambda\epsilon_{2})$, are one-dimensional 
where $\lambda\in {\mathbb Z}$ and the projective cover, $P(\lambda)$, of $L(\lambda)$ is four dimensional. Therefore, 
$$z_{\F}(M) = r\left(\Ext^{\bullet}_{(\g,\g_{0})}(M,\bigoplus_{\lambda\in {\mathbb Z}}L(\lambda))\right)=
r\left(\Ext^{\bullet}_{(\g,\g_{0})}(M,\bigoplus_{\lambda\in {\mathbb Z}}L(\lambda)^{\dim P(\lambda)})\right)=c_{\mathcal F}(M).$$ 

Next observe that $\ff\cong \mathfrak{sl}(1|1)$ which is an ideal in $\mathfrak{gl}(1|1)$ with $\mathfrak{gl}(1|1)/\mathfrak{sl}(1|1)\cong \mathfrak{h}$ where 
${\mathfrak h}$ is the one-dimensional subalgebra spanned by the $2\times 2$ diagonal matrix $\text{diag}\{1,-1\}$. One can now consider 
the spectral sequence: 
$$E_{2}^{i,j}=\text{Ext}^{i}_{({\mathfrak h},{\mathfrak h})}({\mathbb C},\text{Ext}^{j}_{(\f,\f_{0})}(M,\oplus_{\lambda\in {\mathbb Z}}L(\lambda)^{\oplus 4})\Rightarrow 
\text{Ext}^{i+j}_{(\g,\g_{0})}(M,\oplus_{\lambda\in {\mathbb Z}}L(\lambda)^{\oplus 4}).$$ 
This spectral sequence collapses and yields the following isomorphisms: 
\begin{eqnarray*}
\text{Ext}^{\bullet}_{(\g,\g_{0})}(M,\oplus_{\lambda\in {\mathbb Z}}L(\lambda)^{\oplus 4})
&\cong & \text{Hom}_{\mathfrak h}({\mathbb C},\text{Ext}^{\bullet}_{(\ff,\ff_{0})}(M,\oplus_{\lambda\in {\mathbb Z}}L(\lambda)^{\oplus 4})) \\
&\cong &\oplus_{\lambda\in {\mathbb Z}} \ \text{Hom}_{\mathfrak h}({\mathbb C},\text{Ext}^{\bullet}_{(\ff,\ff_{0})}(M,{\mathbb C}^{\oplus 4})\otimes (2\lambda))\\
&\cong &\oplus_{\lambda\in {\mathbb Z}} \ \text{Ext}^{\bullet}_{(\ff,\ff_{0})}(M,{\mathbb C}^{\oplus 4})_{-2\lambda}\\
&\subseteq &\text{Ext}^{\bullet}_{(\ff,\ff_{0})}(M,{\mathbb C}^{\oplus 4}).
\end{eqnarray*} 
The lower subscript $-2\lambda$ on the third line indicates the $-2\lambda$ weight space under the action of the aforementioned matrix in ${\mathfrak h}$. 
Now we use the following facts: (i) $M$ is in the principal block for ${\mathcal F}_{(\g,\g_{0})}$ thus is in the principal block of 
${\mathcal F}_{(\ff,\ff_{0})}$, (ii) the principal block of ${\mathcal F}_{(\ff,\ff_{0})}$ has one simple module (namely the trivial module), 
and (iii) the projective cover of the trivial module in ${\mathcal F}_{(\ff,\ff_{0})}$ is four dimensional. These facts in conjunction with the above calculation show that 
$$c_{{\mathcal F}_{(\g,\g_{0})}}(M)\leq c_{{\mathcal F}_{(\ff,\ff_{0})}}(M).$$ 
In order to show equality, one can use the fact that any projective resolution 
in ${\mathcal F}_{(\g,\g_{0})}$ will restrict to a projective resolution in ${\mathcal F}_{(\ff,\ff_{0})}$. Finally, we apply \cite[Theorem 2.9.1]{BKN3}, and the 
fact that the there is only one simple module in the principal block of ${\mathcal F}_{(\ff,\ff_{0})}$ to conclude that 
$c_{{\mathcal F}_{(\ff,\ff_{0})}}(M)=\dim {\mathcal V}_{(\ff,\ff_{0})}(M)$. 
\end{proof}

\let\section=\oldsection
\bibliographystyle{eprintamsmath}
\bibliography{BKN4}

\end{document}